\documentclass{amsart}

\usepackage{mathstuff}
\usepackage{csquotes}

\let \ssn \subsurfacenorm

\MakeOuterQuote{"}
\DeclareMathOperator{\diam}{diam}

\begin{document}
\sloppy 
\title{Simple Connectivity of Spheres in the Curve Complex}
\author{Richard Cao and Rishibh Prakash}
\begin{abstract}
    For a fixed radius $r$ and a point $o$ in the curve complex of a surface, we define the sphere of radius $r$ to be the induced subgraph on the set of vertices of distance $r$ from $o$. We show that these spheres are almost simply connected for surfaces of high enough complexity, in the sense that loops in the sphere bound an embedded disk contained in a small neighborhood of the sphere.
\end{abstract}
\maketitle
    
\setcounter{tocdepth}{1}
\tableofcontents
    
\section{Introduction}
\subsection{Context and Previous Results}

The complex of curves for a compact Riemann surface of genus $g$ with $n$ marked points encodes intersection patterns of essential simple closed curves. It is a simplicial complex where the vertices are defined to be simple closed curves, and two vertices share an edge when the curves have disjoint representatives. It was first introduced by Harvey \cite{Harvey} and is a central object in Teichmüller theory, geometric group theory, and low-dimensional topology. Since the introduction of the curve complex, one of the most influential results on the curve complex has been the theorem of Masur and Minsky~\cite{MMI} establishing that it is Gromov hyperbolic. 
Its combinatorial structure also reflects deep properties of the mapping class group~\cite{MMII}. Furthermore, Minsky uses it to study the ends of hyperbolic 3-manifolds and prove Thurston's Ending Lamination Conjecture ~\cite{ELC}. 

In a Gromov hyperbolic space, the boundary at infinity is defined by equivalence classes of geodesic rays, and has a metric that records how geodesics diverge. The hyperbolicity of the curve complex gives it a boundary at infinity, making it a natural object to study. The topology of the boundary of curve complexes is well studied : Klarreich~\cite{klarreich2018boundary} shows that it is homeomorphic to ending lamination space, and Gabai shows it is $(n-1)$-connected where $n$ is a function of the genus and number of  punctures~\cite{Gabai1,Gabai2}. Continuing this study of the boundary, Alex Wright~\cite{wright2024} reproves and improves the result that the boundary is connected using that spheres in the curve complex are connected once the underlying surface has sufficiently high complexity. 

In this paper, we build upon Wright’s result by showing that spheres in the curve complex are almost simply connected for surfaces of high enough complexity. Our approach combines tools from Teichm\"uller theory with explicit geometrical cases. In particular, we make frequent use of the Bounded Geodesic Image Theorem to control distances in the curve complex, and we work directly in the curve complex by constructing and analyzing explicit families of curves on the underlying surface.

\subsection{Main Results}

Let $\Sigma=\Sigma_{g,n}$ be a connected surface with genus $g$ and $n$ punctures. We always assume $g$ and $n$ are such that $\Sigma$ has a complete, finite volume hyperbolic metric and that $(g,n) \neq (0, 3)$. We define the complexity of $\Sigma$ as $\xi(\Sigma) = 3g - 3 + n$. A curve on $\Sigma$ is the homotopy class of a simple closed curve. Let $C\Sigma$ be the curve complex of $\Sigma$, which is the simplicial complex with $n$-simplices corresponding to $n+1$ curves that can be realized disjointly, with face relation given by inclusion. Fix an arbitrary vertex $o\in C\Sigma$. For $r\geq0, ~r\in\ZZ$, define $S_r = S_r(o)$ be the sphere of radius $r$ in $C\Sigma$, which is the subcomplex induced by all the vertices distance $r$ away from $o$.

We will prove the following theorem.
\begin{theorem}
    Let $\xi(\Sigma)\geq 10$. There is some $M$ depending only on $\Sigma$ such that for any embedded loop $\gamma \subset S_r$ there exists an embedded disk $\Delta \subset B_{r+M}\setminus B_{r-3}$ with boundary $\gamma$.
\end{theorem}

\begin{remark}
    We expect the proof of this theorem to generalize to higher connectivity, as in for each $k$, there is some $n,M,N$ such that for surfaces of complexity greater than $n$, for any $r$ and any embedded $k$-sphere in $S_r$, there exists an embedded $(k+1)$-ball in  $B_{r+M}\setminus B_{r-N}$ whose boundary is the $k$-sphere. 
\end{remark}

Also, the bound we obtain for the complexity of the surface in the theorem is not necessarily strict, we also expect it to be improved: 

\begin{question}
Can the same conclusion hold for surfaces of lower complexity, such as the six times punctured sphere?
\end{question}

Wright shows that the boundary of the curve complex is linearly connected, which means there is some $K \in \RR$ such that for any $x,y \in \partial C\Sigma$, the Gromov boundary of the curve complex, there is a compact connected set in $\partial C \Sigma$ containing $x,y$ with diameter at most $Kd(x,y)$. The result follows from using the connectivity of spheres to inductively build a connected path in the boundary. We expect that this proof can be modified to show that the boundary of the curve complex of surfaces of high enough complexity are linearly simply connected:

\begin{conjecture}
    For $\xi(\Sigma)\geq 10$, the Gromov boundary of the curve complex $\partial C\Sigma$ is linearly simply connected, which is to say there exists some constant $K$ depending on $\Sigma$ such that if $f: S^1 \to \partial CS$ is continuous then there is some continuous $\bar{f}: D^2 \to \partial CS$ with $\bar{f}|_{S^1}=f$ and $\diam(\bar{f}(D^2)) \leq K \diam(f(S^1))$.
\end{conjecture}


\subsection{Acknowledgments} We thank Kasra Rafi for his guidance on this paper.

\section{Sufficient Conditions for Connectivity}
    
We recall and rephrase some definitions from Wright~\cite{wright2024}. For the following definitions, fix $\Sigma=\Sigma_{g,n}$ and let $\Upsilon_{h, m, p, u} = \Upsilon\subset \Sigma$ be a subsurface. Here $h$ is the genus of $\Upsilon$, $m$ is the number of punctures it contains, $p$ is the number of paired boundaries and $u$ is the number of unpaired boundaries. The subsurface $\Upsilon$ is meant to be thought of as a connected component of the complement of a multicurve on $\Sigma$. The paired boundaries are understood as boundaries of $\Upsilon$ that correspond to the same curve when considered as a curve in $\Sigma$, and the unpaired ones are boundary components that do not. Let $\hat{\Upsilon}$ denote the surface of genus $h+p$ with $m$ punctures and $u$ boundary components obtained by gluing each of the $p$ pairs.

\begin{figure}[h!]
\centering
\includegraphics[width=1\linewidth]{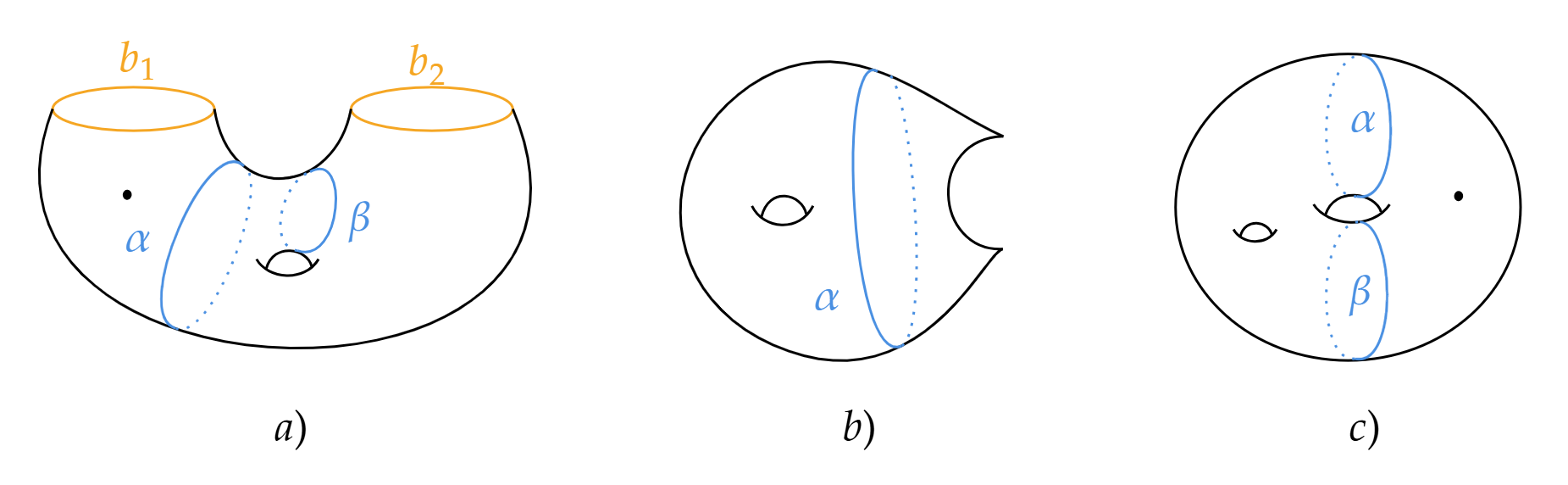}
\caption{a) An eventually non-separating multicurve, b) a pants curve, c) an essentially non-separating multicurve}
\label{examples}
\end{figure}

We call $\alpha$ a multicurve on $\Upsilon$ when $\alpha$ is a collection of curves on $\Upsilon$ that can realized disjointly.  A multicurve on $\Upsilon$ is called \emph{eventually non-separating} if it does not separate $\hat{\Upsilon}$. A curve $\alpha$ on $\Upsilon$ is called a \emph{pants curve} if it bounds a genus 0 surface with 2 punctures. A curve will be called \emph{essentially non-separating} if it is either eventually non-separating or a pants curve. A multi-curve $\alpha\cup \beta$ with two non-isotopic components will be called \emph{essentially non-separating} if $\alpha$ and $\beta$ are each essentially non-separating, and either
\begin{enumerate}
\item $\alpha\cup \beta$ is eventually non-separating, or
\item at least one of $\alpha$ or $\beta$  is a pants curve, or
\item or $\alpha\cup \beta$ bounds a genus zero subsurface with no boundary components and 1 puncture. 
\end{enumerate}
Equivalently, there is an edge between $\alpha$ and $\beta$ if $\hat{\Upsilon} \setminus (\alpha, \beta)$ has at most one component which is not homeomorphic to a thrice punctured sphere.

For a cut surface $\Upsilon=\Upsilon_{h,m,p,u}\subset \Sigma$, the essentially non-separating graph $C_0\Upsilon$ is the subgraph of $C\Upsilon$ whose vertices are essentially non-separating, with an edge from $\alpha$ to $\beta$ if $\alpha \cup \beta$ is essentially non-separating.


\begin{definition}
    For any multicurve ($k$-simplex) $X=\{x_0,x_1, \dots, x_k\}\subset C\Sigma$, $\Sigma \setminus (\bigcup_{i=1}^k x_{i})$ is a disjoint collection of connected subsurfaces $\{U_j\}_{j=1}^m$. If there is a unique $l$ such that $U_l$ is not homeomorphic to $\Sigma_{0,3}$, $U_l$ is called the \emph{unique non-pants} of $\Sigma\setminus X$.
\end{definition}

\begin{definition}
    Let $\Upsilon\subset V\subset \Sigma$ be subsurfaces. A $k$-simplex $X=\{x_0,x_1,...x_k\}\subset C\Upsilon$ is called \emph{good} in $V$ if for any subsimplex $\{x_{i_1},...x_{i_j}\}\subset X$, $V\setminus \{x_{i_1},...x_{i_j}\}$ has a unique non-pants. If $V=\Sigma$ then $X$ is simply called good.
    The full subcomplex in $C\Upsilon$ consisting of all good simplices in $V$ we denote as $(\Upsilon,V)$.
\end{definition}

We take some time to explore the relationship between the good subcomplex $(\Upsilon, \Sigma)$ and essentially non-separating subcomplex $C_0\Upsilon$. We note that good vertices are in $C_0$, and that $C_0$ is contained in the subcomplex of good edges in the cases where there are few enough boundaries.

\begin{proposition}\label{vertsgood}
    Suppose $\Upsilon=\Sigma\setminus X$ where $X$ is a vertex or an edge and $\xi(\Sigma)\geq 3$. Then good vertices are in $C_0\Upsilon$.
\end{proposition}

\begin{proof}
    Take $a$ a good vertex. If $a$ is a pants curve then it is in $C_0\Upsilon$. If $a$ is non-separating on $\Sigma$, we show that it is eventually non-separating.
    
    Suppose $X = \{x, y\}$. If both of $x,y$ are separating, $\Upsilon$ has no paired boundaries and so we want to show that $a$ is non-separating on $\Upsilon$. So $\Sigma\setminus(x,y)$ is 3 components and $a$ must lie in one of the components and be non separating on that component since it is disjoint from $x,y$, so it is non-separating on $\Upsilon$.
    
    If exactly $x$ is non separating, say $y$ separates $\Sigma$ into components $L,R$ and $x$ lies in $R$. If $a$ lies in $L$ we have that $L$ has no paired boundaries so $a$ is eventually non-separating. If $a$ lies in $R$ then we have that $R$ has a paired boundary coming from $x$ and an unpaired one from $y$, so upon gluing the boundaries from $x$ we get that $a$ is eventually non-separating. If $x,y$ are both non-separating, we either get that $\Upsilon$ has 2 sets of paired boundaries in which case gluing them gives $\hat{\Upsilon}=\Sigma$ so $a$ is eventually non-separating, or $\Upsilon$ has one set of paired boundaries, where since $a$ is non-separating on $\Sigma$ means that it is also non-separating on $\hat{\Upsilon}$ as we wanted.
\end{proof}

\begin{proposition}\label{translate}
    Suppose $\Upsilon=\Sigma\setminus X$ where $X$ is a vertex or an edge and $\xi(\Sigma)\geq 3$. Then all the edges of $C_0\Upsilon$ are good.
\end{proposition}

\begin{proof}
    Pick a vertex $\alpha\in C_0\Upsilon$. $\alpha$ is either a pants curve or eventually non-separating. If it is a pants curve on $\Upsilon$, it is a pants curve on $\Sigma$ and so cuts $\Sigma$ into at most one non-pants component. If it cuts $\Sigma$ into only pants components, it means $\Sigma=\Sigma_{0,4}$ or $\Sigma_{1,1}$ and so $\xi(\Sigma)=1$.
    
    If $\alpha$ is eventually non-separating, let $\hat{\Upsilon}$ be the surface of genus $h+p$ with $m$ cusps and $u$ boundary components obtained by gluing all the $p$ pairs. Since $\hat{\Upsilon}\subset \Sigma$, and $a$ is non separating on $\hat{\Upsilon}$, it is non separating on $\Sigma$ so $a$ is good in $\Sigma$. This shows that the vertex set of $C_0\Upsilon$ is contained in the vertex set of $(\Upsilon, \Sigma)$. We check the same holds true for the edges.

    Pick an edge $\{a,b\}\subset C_0\Upsilon$. Each of $a,b$ are essentially non-separating so are good. Now either $a\cup b$ is eventually non-separating, or at least one of $a$ or $b$ is a pants curve, or $a\cup b$ bounds a genus zero subsurface with no boundary components and 1 puncture. In the first case, $a\cup b$ does not separate $\hat{\Upsilon}$ so does not separate $\Sigma$. Then $\Sigma\setminus \{a,b\}$ is not a pants since $\xi(\Sigma\setminus\{a,b\})=\xi(\Sigma)-2\geq 1$. In the second case, if say $a$ is a pants curve and $b$ is a pants, then $\Sigma\setminus\{a,b\}$ consists of two pants and one other component, which cannot be a pants because if it were, $\Sigma=\Sigma_{0,5}$ and $\xi(\Sigma)=2$. If $a$ is a pants curve and $b$ does not separate $\hat{\Upsilon}$ then $\Sigma\setminus\{a,b\}$ is one pants and one non-pants. In the third case, $\Sigma\setminus \{a,b\}$ is one pants and one non-pants, so the edge is good.

\end{proof}

We will call a 3-simplex a tetrahedron, a 2-simplex a triangle, a 1-simplex an edge, a 0-simplex a vertex, and a union of two 2-simplices that intersect at exactly one edge a diamond. Our notation for neighborhoods and distances is $B_r(x)=\{y, ~d(x,y)\leq r\}$, $B_r(X)=\cup_{x\in X} B_r(x)$, $d(x,Y)=d(Y,x)=\inf_{y\in Y}d(x,y)$, $d(X,Y)=\sup_{x\in X} d(x,Y)$, and if an origin $o$ is fixed, $B_r=B_r(o)$. Moreover if an origin $o$ has been fixed, for any subsurface $V\subset \Sigma$ we write $\ssn{x}= d_V(x,o)$.

\section{Key Lemmas}\label{lemmas}

\begin{theorem}\label{ozsc} 
    Let $\xi(\Sigma)\geq 10$. Suppose $B_R$ is $M$-almost simply connected for all $B_R\subset C\Sigma$ (see \autoref{asc}). Then for any $\gamma\subset S_r$ there exists a triangulation $\Delta'\subset B_{r+M+7}\setminus B_{r-3}$ of $\gamma$.
\end{theorem}
This theorem follows from \autoref{ozoz} and \autoref{simcon}. We prove ~\autoref{ozoz} in this section. The idea of the proof is that, we can use simple connectivity of the curve complex to triangulate a loop and then we may push simplices in the triangulation away from the origin in three cases. The three cases are a triangle (\autoref{triangle}), a diamond (\autoref{diamonds}), and a cone (\autoref{cone}), as shown in ~\autoref{indpic}.

The main tool in showing that spheres are connected is the Bounded Geodesic Image Theorem from \cite{MMI} which we restate here.

\begin{theorem}\label{bgit}
    For every connected surface $U$, there exists a constant $M$ such that if $V$ is a subsurface of $U$, and $a,b \in CU$, and $d_V(a,b) > M$, then every geodesic from $a$ to $b$ in $CU$ contains a curve which does not cut $V $.
\end{theorem}

\begin{figure}[h]
\centering
\includegraphics[width=1\linewidth]{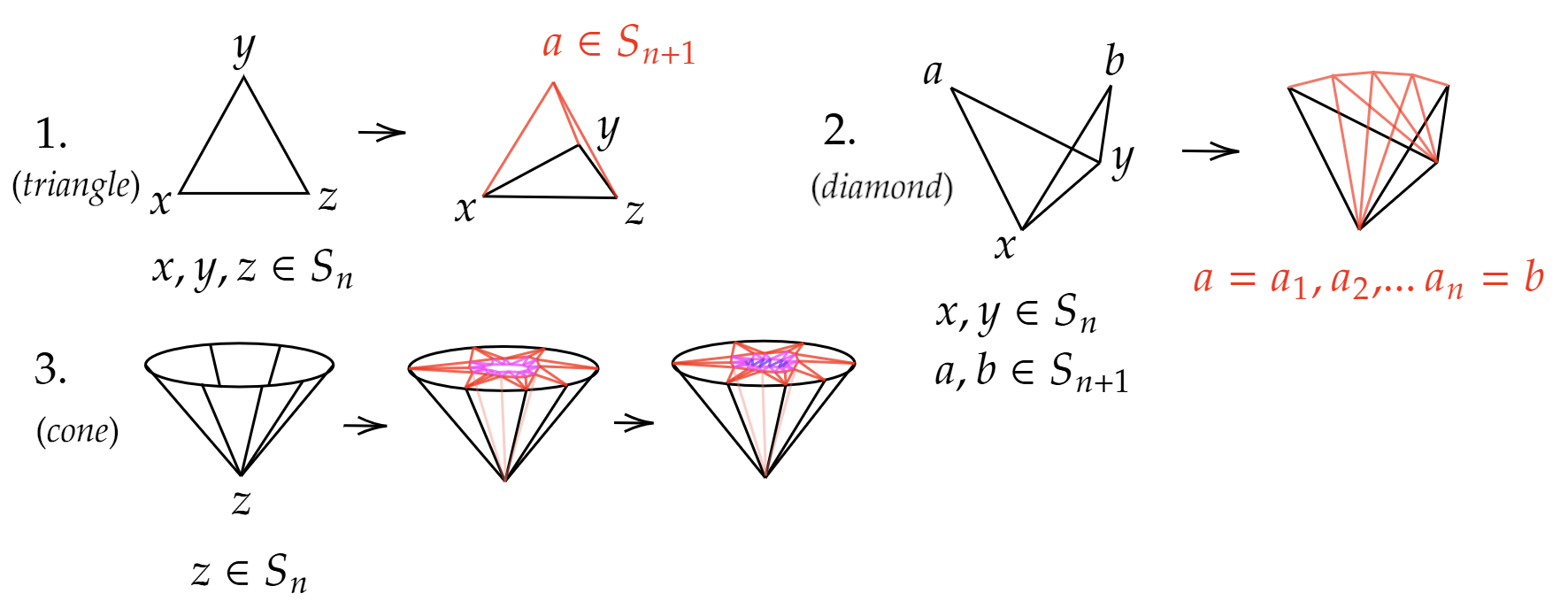}
\caption{Pushing simplices away from the origin}
\label{indpic}
\end{figure}

\subsection{Triangle}\label{triangle}

This case is one where we have three vertices $x,y,z$ which form a triangle in our triangulation. We want to introduce another vertex $w$ so that the edges $(x,y),(y,z),(x,z)$ are not affected but the triangle $(x,y,z)$ is replaced with a collection of triangles $(w,x,y),(w,x,z),(w,y,z)$.

\begin{proposition}\label{inductionlem1}
Fix $\Sigma=\Sigma_{g,n}$ with $\xi(\Sigma)\geq 5$. Fix an origin $o\in C\Sigma$ and $r\in\ZZ_{\geq 0}$. If $(x,y,z)$ is a good triangle in $S_r\cup S_{r+1}$ where $x\in S_r$, then there exists some $w\in S_{r+1}$ such that $(w,x,y,z)$ is a tetrahedron with good faces. Furthermore, the set of these $w$ is coarsely dense in $C\Upsilon$ where $\Upsilon$ is the unique component of $S\setminus(x\cup y\cup z)$ that is not a pants.
\end{proposition}

\begin{proof}
We first show the density statement follows from the mapping class group action. Let $M$ be the constant from \autoref{bgit} and suppose we have $w\in C\Upsilon$ such that $(w,x,y,z)$ is a tetrahedron with good faces. The orbit $G$ of $w$ under the action of $Mod(\Sigma)$ is coarsely dense in $C\Upsilon$, say it's $d$-dense. By \autoref{bgit} it follows that any point $v\in G\setminus B_M(o)$ is such that $v,x,y,z$ is a tetrahedron with good faces and $v\in S_{r+1}$. But now $G\setminus B_M(o)$ is coarsely dense.

In order to show existence of $w$, we will begin by finding $a$ so that $(a,x, y, z)$ form a good tetrahedron and then modifying $a$ so that it lies in $S_{r+1}$.

Suppose $x,y,z$ form a good triangle. Then $\Upsilon=\Upsilon_{h,m,b}$ the unique non-pants has complexity $\xi(\Upsilon)=\xi(\Sigma)-3\geq 2$. If $h\geq 1$, we let $a$ be non-separating curve in $C\Upsilon$. If $h=0$ then $m+b\geq 5$. If $m\geq 2$, take $a$ to be a pants curve around two punctures. If $m=1$ then $b\geq 4$ so one of $x,y,z$, say $x$, is non-separating on $\Sigma$. Then choose $a$ to be a curve that makes a pants with $x$ and the puncture. This can be done for example by connecting the puncture to a boundary formed by $x$ with an arc and taking the pants curve associated to that arc. This makes good faces since $a\in (\Upsilon,\Sigma)$, $(a,x)\subset  (\Upsilon,\Sigma)$ since they cut out a single pants, $(a,y),(a,z) \subset (\Upsilon,\Sigma)$ since $y,z$ are either pants curves or jointly non-separating with $x$, in which case they are still jointly non separating with $a$, and finally $(a,x,y),(a,x,z),(a,y,z)\subset (\Upsilon,\Sigma)$.

If $m=0$ then $b\geq 5$ and so there are at least two non-separating curves that are also jointly non-separating, say $x,y$. Then take $a$ to be the pants curve associated to an arc that connects a boundary resulting from $x$ and one from $y$.

If $\ssn[\Upsilon]{a} > M$, then we can take $w = a$ because \autoref{bgit} implies that any geodesic from $w$ to $o$ must miss $\Upsilon$. Since no curve other than $x,y,z$ miss $\Upsilon$, we get that $w \in S_{r+1}$.

If $\ssn[\Upsilon]{a} \leq M$, then we take $g$ to be a pseudo-anosov map on $\Upsilon$. We know such maps have positive translation distance on $C\Upsilon$. Thus we can take $w = g^k a$ where $k$ is large enough so that $\ssn[\Upsilon]{g^k a} > M$.
\end{proof}

The above requires us to work with good triangles. However when triangulating a loop in the curve complex, it may be that some of the triangles are bad. The following two statements allows us to modify the triangulation to replace bad triangles with good ones. This will be a recurring theme where we will frequently have to fix vertices, then edges and then triangles of some given triangulation to be able to apply the Bounded Geodesic Image Theorem and push it outwards.

\begin{proposition}[Replacing a bad triangle]\label{fixinglem}
Let $\Sigma = \Sigma_{g, n}$ be a surface with $\xi(\Sigma)\geq 5$.
Suppose $x \in S_r$ and $\{x,y,z\}\subset S_r\cup S_{r+1}$ form a triangle with good edges but $(x,y,z)$ is not a good triangle. Then there exists some $a\in S_r\cup S_{r+1}$ such that $(a,x,y), (a,x,z), (a,y,z)$ are good triangles in $S_r\cup S_{r+1}$. If $x=o$, then $a$ can be taken to be in $S_1$.
\end{proposition}

\begin{figure}[h]
    \centering
\includegraphics[width=.5\linewidth]{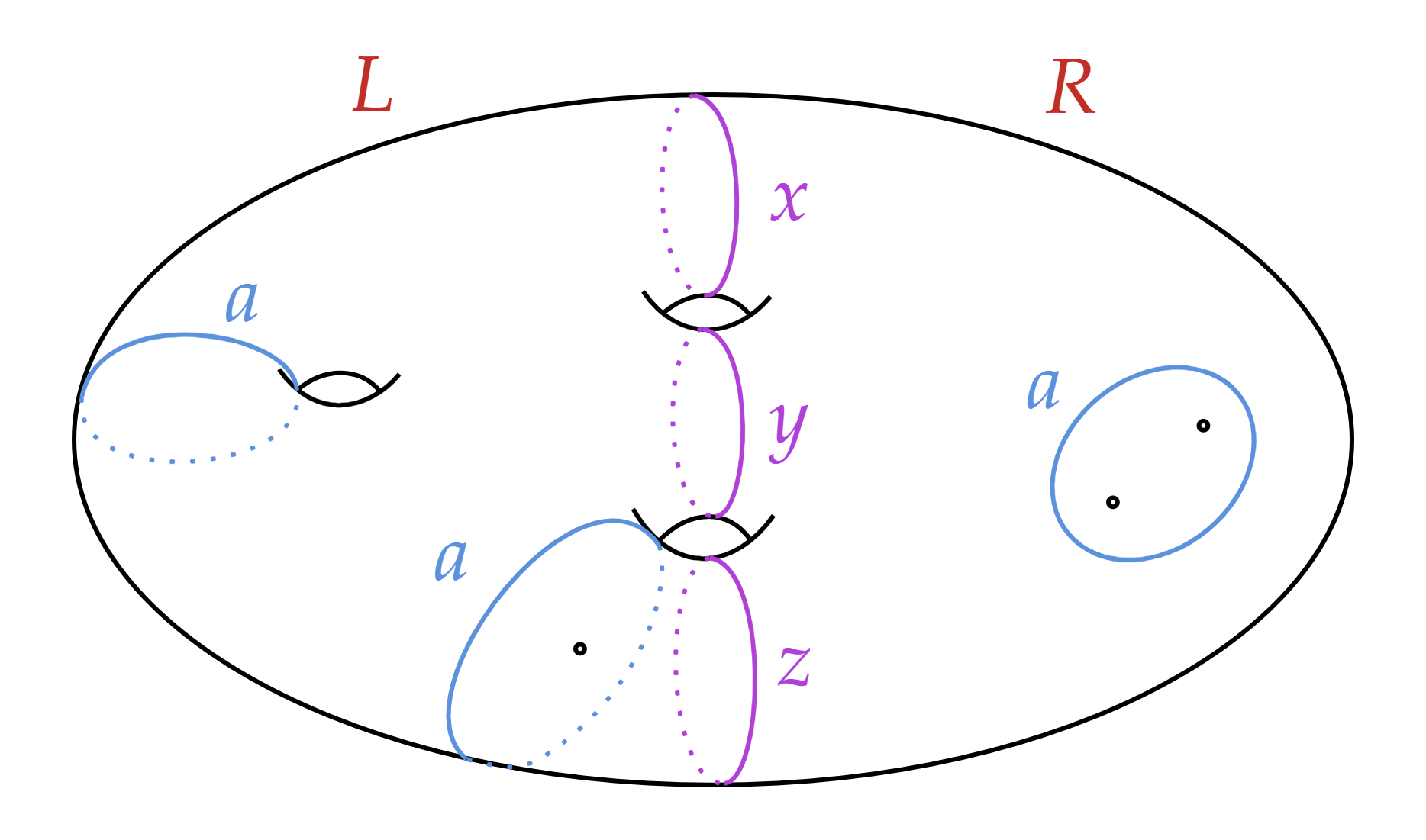}
    \caption{Picking a non-separating or pants curve on $S\setminus (x\cup y\cup z)$}
    \label{badtriple}
\end{figure}

\begin{proof}
We have a bad triangle $(x,y,z)$ with good edges in $S_r\cup S_{r+1}$, and $x\in S_r$. This means each pair in $x,y,z$ is jointly non-separating while $x\cup y\cup z$ is separating and does not cut out a pants. Up to homeomorphism, the curves look like those in \autoref{badtriple}. Let $L,R$ be the two components which are both not pants. Notice this means there is at least a genus or a puncture in the interior of $L$ and $R$.

Say $x=o$. If $g\geq 3$, then one of $L$ or $R$ has positive genus and we can pick a non-separating there. If $g=2$ then $n\geq 2$. If $n\geq 3$ then one side has at least 2 punctures so take $a$ to be a pants curve on one side. If $n=2$ then take $a$ to be any curve on one side. Then $(a,y,z),(a,x,z),(a,x,y)$ are good triangles and $d_\Sigma(a,o)=1$.

So suppose none of $x,y,z$ are the origin. Then $o$ must cut at least one of the components $L$ or $R$, say $o$ cuts $L$. If $\xi(L)= 1$, pick $a$ to be any curve in $CL$ as in the bottom curve of Figure \autoref{badtriple}. Then $(a,x,y),(a,x,z),(a,y,z)$ are good triangles. 

Let $L=L_{h,m,b}$, so $b=3$. If $\xi(L)\geq 2$, we have $3h+m\geq 2$. If $m\geq 2$, take $a$ to be a a pants curve on $L$, 
If $h\geq1$, take $a$ to be a non-separating curve on $L$. In any case, the triangles that $a$ forms with $x,y,z$ will be good. 

By the mapping class group action on $CL$, replace $a$ by a homeomorphic curve so that $d_L(a,o)>M$. Pick any geodesic from $a$ to the origin in $C\Sigma$, say $o=a_0,a_1,...a_n=a$. By \autoref{bgit} we know there is some curve on this geodesic that misses $L$. Since $a$ doesn't miss $L$, suppose $a_i$ misses $L$. So $a_i\in\{x,y,z\}$ or $a_i\in CR$. If $a_i\in\{x,y,z\}$, $a\in S_{r+1}$. If $a_i$ is not one of $x,y,z$, then $a_i$ is disjoint from $x,y,z$ since it is in $R$. This means
$$\ssn[\Sigma]{a_i} \geq \min\{\ssn[\Sigma]{x}, \ssn[\Sigma]{y}, \ssn[\Sigma]{z}\}-1=r-1$$

Then $\ssn[\Sigma]{a}\geq \ssn[\Sigma]{a_i} + 1\geq r$. In any case, $a\in S_r\cup S_{r+1}$ and $(x,y,a), (x,z,a), (y,z,a)$ are good triangles.
\end{proof}

By combining the above propositions we get the following.

\begin{lemma}\label{replacebad1}
Let $\xi(\Sigma)\geq 5$. Suppose $(a,x,y)$ is a bad triangle with good edges, with $x\in S_r$, $y\in S_r\cup S_{r+1}$. Then either there exists $c\in S_r,d\in S_{r+1}$ such that $$(a,c,x), (c, x, y), (a,c,y), (c,x,d),(c,y,d),(d,x,y)$$ are good triangles, or there is $c\in S_{r+1}$ such that $(a,c,x),(c,x,y),(a,c,y)$ are good triangles.
\end{lemma}

\begin{proof}
Suppose $(a,x,y)$ is not a good triangle but has good edges. 
Then we may apply \autoref{fixinglem} on triangle $(a,x,y)$ to obtain a $c\in S_r\cup S_{r+1}$ such that $(c,a,x),(c,a,y),(c,x,y)$ are good triangles. If $c\in S_{r+1}$ then we are done.

If $c\in S_r$, apply \autoref{inductionlem1} to the triangle $(c, x, y)$ obtain a $d\in S_{r+1}$ such that $(c,x,y,d)$ is a tetrahedron with good faces, and so we are done. The process is pictured in \autoref{replacebaddiamond}.
\end{proof}

\begin{figure}[h]
\centering
\includegraphics[width=.7\linewidth]{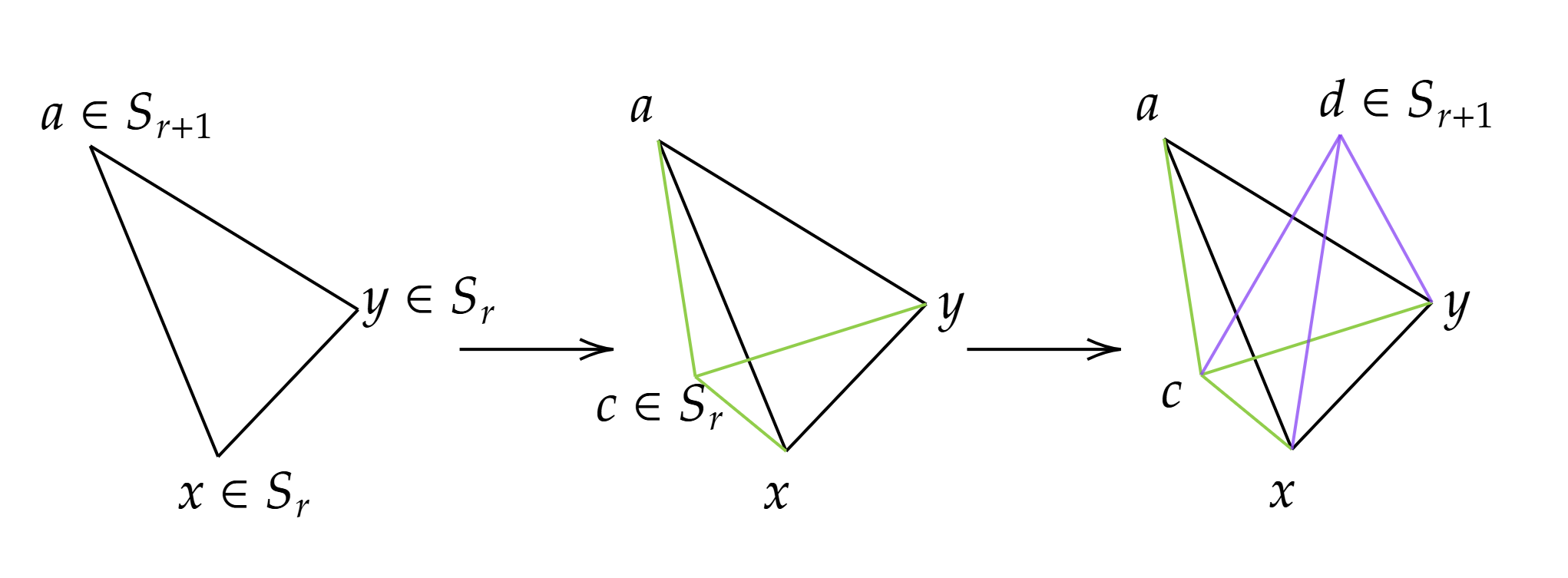}
\caption{replacing bad triangles}
\label{replacebaddiamond}
\end{figure}

\subsection{Diamond}\label{diamonds}

This section deals with pushing away diamonds. The quintessential example in this case is where $(x, y)$ is a good edge in a triangulation in $S_r$ and we have applied \autoref{inductionlem1} to the triangles containing $(x, y)$ in order to find vertices $a, b \in S_{r+1}$. Thus now $(a, x, y, b)$ form a diamond. We want to remove the edge $(x,y)$ while keeping the loop $(a,x),(a,y),(b,x),(b,y)$ and find a new triangulation for this loop.

The main proposition of this section is \autoref{inductionlem2} which will be proven using a collection of lemmas. 

\begin{proposition}[diamond]\label{inductionlem2}
Fix a surface $\Sigma$ with $\xi(\Sigma)\geq 8$. Let $a,b\in S_{r+1}, ~x\in S_r,~y\in S_r\cup S_{r+1}$, and $(a,x,y),(b,x,y)$ are good triangles. Then there exists a path $a=v_0,...v_p=b\in S_{r+1}$ such that all the $(v_i,v_{i+1},x)$, $(v_i,v_{i+1},y)$ are good triangles.
\end{proposition}

The rough outline of the proof is to work in the curve complex $CV$ where $V$ is the unique non-pants component of $\Sigma \setminus (x, y)$. We would like to apply Wright's connectivity results of $C_0$ to find the $v_i$. The slight wrinkle involved with doing this immediately is we also wish to apply the Bounded Geodesic Image Theorem, so we need to ensure that the path in $CV$ is at least $M$ away from the origin. Our solution for this is to begin by pushing $a, b$ to the complement of a very large ball in $CV$ (this is \autoref{connecttoO1}) and connecting them in the complement of a slightly smaller ball so that we always remain at least $M$ away in $CV$ itself (this is \autoref{goodishpath}). Finally it may be that the path we find, although has good edges within itself, may form bad edges with $x$ or $y$. Thus the final step is modify the path to fix this. This is \autoref{fixpath}.



\begin{lemma}\label{connecttoO1}
Let $\xi(\Sigma)\geq 5$. Let $(a,x,y)$ be a good triangle with $a\in S_{r+1}, x \in S_r ,y\in S_r\cup S_{r+1}$. Let $V$ be the unique non-pants of $\Sigma \setminus (x,y)$. Then for any $C$, there is a path $a=x_0, x_1, ...,  x_n\in S_{r+1}$ such that $\ssn{x_n} \geq C$, and for all $i,(x_i,x_{i+1},x,y)$ are tetrahedra with good edges.
\end{lemma}

\begin{proof}
Let $a,x,y$ be a good triangle and $U$ be the unique non-pants component of $\Sigma \setminus (a, x, y)$. Then, by \autoref{inductionlem1}, for any $M$ we can find a $w\in S_{r+1}$ such that $w,a,x,y$ is a good tetrahedron, with $\ssn[U]{w}>M$. Now since $U$ is a subsurface of $V$, picking $M$ large enough \autoref{bgit} implies that any geodesic from $w$ to $\pi_V(o)$ in $CV$ must miss $U$. However, the only curve that misses $U$ while staying in $V$ is $a$, so we get that $\ssn{w} \geq \ssn{a} + 1$. Since $w$ is disjoint from $a$, we actually get equality: $\ssn{w} = \ssn{a} + 1$.

But now we have that $w, a, x, y$ is a tetrahedron with good faces, and $w,x,y$ is a good triangle, so we may repeat this procedure on $w,x,y$ in place of $a,x,y$ to obtain $w'$ such that $\ssn{w'} = \ssn{w} + 1$. Iterating this $C$ times, we get a path $a = w_0, w_1, ..., w_C\in S_{r+1}$ with $\ssn{w_C} = \ssn{a} + C \geq C$, and all the $(x_i,x_{i+1},x,y)$ are tetrahedra with good faces.

\end{proof}

\begin{figure}[h]
\centering
\includegraphics[width=1\linewidth]{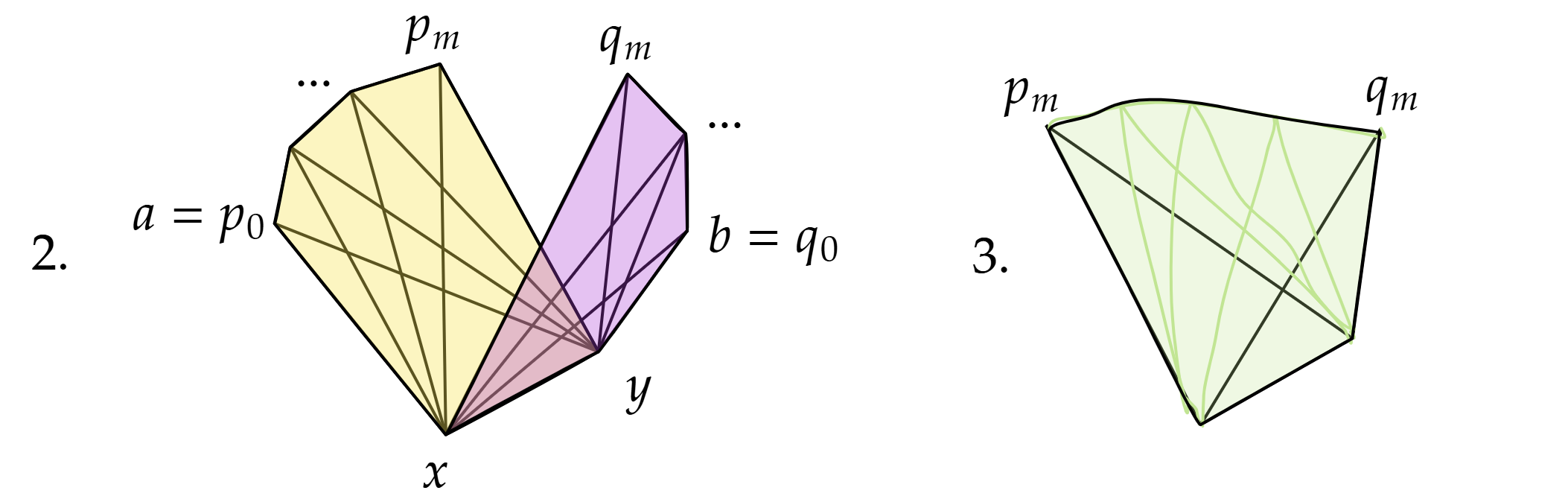}
\caption{Plan for diamond case}\label{badcone}
\end{figure}

    

\begin{lemma}\label{findgoodedge}
    Let $V$ be a subsurface with 2 boundaries of $\Sigma$ such that $\xi(V) \geq 3$. Let $a$ be in $(V, \Sigma)$. Then there exists $a' \in (V, \Sigma)$ which form a good edge with $a$.
\end{lemma}
\begin{proof}
    Let $\bar V_{\bar h, \bar m, \bar b} = V \setminus a$ be the unique non-pants component. If $\bar V$ has a genus, then we can take $a'$ to be a non-separating curve going around the genus.
    
    So suppose $\bar V$ is a genus 0 surface. We know $\xi(\bar V) \geq 2$ so $\bar b + \bar m \geq 5$. Thus either $\bar m \geq 2$ or $\bar b \geq 4$. In the former case we can take $a'$ to be a pants curve around the pair of punctures. In the latter case we must have that $a$ is non-separating in $V$ so we can take $a'$ to be the curve represented by an arc in $\bar V$ going from one of paired boundaries to some other boundary.
\end{proof}

\begin{lemma}\label{findgoodtriangle}
    Let $V$ be a surface with 2 boundaries such that $\xi(V) \geq 6$. Let $a, b$ be a good edge in $CV$. Then there exists $c \in CV$ making $(a, b, c)$ a good triangle.
\end{lemma}
\begin{proof}
    Write $V = V_{\tilde h, \tilde m, \tilde b}$. Since $\tilde b = 2$ by assumption and $\xi(V) \geq 6$, we must have $3 \tilde h + \tilde m \geq 8$. If $\tilde h = 0$, then $a$ and $b$ must be pants curves and hence encircle 2 punctures each. We can therefore take $c$ to encircle a pair of remaining punctures. If $\tilde h > 2$, then there must be a subsurface with genus which is disjoint from $a$ and $b$ so we can take $c$ to be a non-separating curve on this subsurface. The same argument works in the case of $\tilde h = 1$ when $a$ and $b$ are both pants curves and the case of $\tilde h = 2$ when at least one of them is a pants curve. If these conditions do not hold, then that means there are punctures remaining so we can take $c$ to be a pants curve instead.
\end{proof}

\begin{lemma}\label{Ozisconnected1}
Fix $\xi(\Sigma)\geq 10$. Let $C$ be very large, $M>0$, and $V$ be the unique non-pants of $\Sigma \setminus (x,y)$.
If $x\in S_r,~y\in S_r\cup S_{r+1},~p_a,p_b\in S_{r+1}\subset C\Sigma$ with $p_a, p_b,x,y$ a good diamond and $\ssn{p_a}, \ssn{p_b} > M + C$ then there is a path $\{p_i\}$ in $CV$ with $\ssn{p_i} > M$ connecting $p_a$ to $p_b$ such that $p_i,p_{i+1},x,y$ are tetrahedra with good edges. 
\end{lemma}
\begin{proof}

We begin by using \autoref{goodishpath} to find a path $\{p_i\}$ in $C_0V$ connecting $p_a$ and $p_b$. If all of these curves 
make good edges with $x$ and $y$, then we are done. If there is some $p_i$ which makes bad edges with $x$ or $y$, we can use \autoref{fixpath} to replace $p_i$ with the path $\{q_j\}$ connecting $p_{i-1}$ to $p_{i+1}$. Since each $q_j$ is at most 2 away from $\{p_{i-1}, p_i, p_{i+1}\}$, we have $\ssn{q_j} > M$. Hence by relabeling the curves, we get a path $\{p_i\}$ connecting $p_a$ and $p_b$ with $\ssn{p_i} > M$ for all $i$.

\end{proof}

\begin{sublemma} \label{goodishpath}
Given the above setup, there exists a path $\{p_i\}$ in $CV$ with $\abs{p_i}_V > M$ connecting $p_a$ to $p_b$ such that each $(p_i,p_{i+1})$ is a good edge.
\end{sublemma}
\begin{proof}
Since $\xi(V)\geq2$ we have a path $p_a=p_1,...p_l=p_b\in CV$ where all the $p_i$ have $|p_i|_V\geq M+C$. This is by \cite{wright2024} which says that unions of spheres in the curve complex are connected for surfaces of complexity at least 2. It is possible that some of these $p_i$ are not good on $\Sigma$. We fix this by finding neighboring curves to these which will be good on $\Sigma$.

Let $p_i$ be not a good curve on $\Sigma$. We have $\xi(\Sigma)\geq 4$.
Let $V=V_{h,m,p,u}$. Since the edge $x,y$ is good, $\xi(V)=\xi(\Sigma)-2\geq 2$. This complexity bound and the fact that $p + u = 2$ (since $V$ is in the complement of two curves) implies that at least one of the conditions of Lemma 3.10 of \cite{wright2024} hold. So every $p_i$ is either in $C_0V$ so good in $\Sigma$ or adjacent to a curve in $C_0V$. Note that $p_a,p_b$ are in $C_0$ by \autoref{vertsgood}.


Now we wish to connect these curves in $C_0V$ using the connectivity of $C_0$. Lemmas 3.12, 3.13 and 3.15 in \cite{wright2024} are a collection of related lemmas telling us when the $C_0$ complex of subsurfaces are connected, depending on the genus, punctures and boundaries in the subsurface. Below we check that the conditions for these lemmas are met.

If $h\geq 2$, \cite[Lemma 3.12]{wright2024} holds to connect these $p_i'$ with paths that stay within a constant of a $CV$-geodesic between them.  
Suppose $h = 1$. We have $\xi(\Sigma)\geq 6$ so $\xi(V)=\xi(\Sigma)-2\geq 4$. If $p=0$ then this means $b = 2$, so then
$$\xi(V)=3h-3+m+b = m+2\geq 4$$

So $m\geq 2$ and \cite[Lemma 3.13]{wright2024} holds, and if $p = 1 = u$ then $m + 2p + u \geq 3$, and if $p=2$ then $m+2p+u\geq 3$ so in any case, \cite[Lemma 3.13]{wright2024} holds. Finally, suppose $h=0$. We have $\xi(\Sigma)\geq 6$. If $p=2$ we have $b=4$, and 
$$\xi(V)=3h-3+m+2p+u=m+p-1\geq 4$$
so $m+p\geq 5$, so \cite[Lemma 3.15]{wright2024} holds. In all of these cases, we get that $C_0 V$ is connected and that for any points in $C_0 V$ for these cases, there exists a $C_0 V$-path connecting them within a bounded distance of a $CV$-geodesic between the points.

Thus assuming $C$ is big enough (in particular bigger than the constants in Lemmas 3.12, 3.13 and 3.15 of \cite{wright2024}), concatenating these paths gives a path where each vertex is of distance at least $M$ from $o$ in $CV$. So we have found a path $\{p_i'\}\subset C_0V\setminus B_V(M,o)$ that connects $p_a,p_b$ where all the edges $p_i',p_{i+1}'$ are good in $\Sigma$ by \autoref{translate}. 
\end{proof}


\begin{figure}[h]
    \centering
    \includegraphics[width=0.3\linewidth]{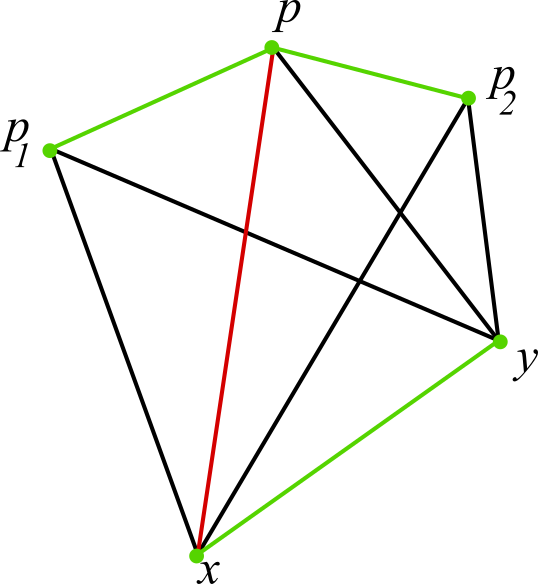}
    \caption{Bad edges will fixed by replacing $p$ with a path making good edges}
    \label{fig:fixpathdiamond}
\end{figure}

\begin{sublemma} \label{fixpath}
    Suppose we have a surface $\Sigma$ and $\{x, y, p_{1}, p_{2}\}$ form a diamond. Suppose there is some $p \in C\Sigma$ disjoint to all of them which makes good edges with $p_{1}$ and $p_{2}$ but bad edges with $x$ or $y$ (see \autoref{fig:fixpathdiamond}). Then there exists a path $\{q_j\}$ connecting $p_{1}$ to $p_{2}$ with $\{q_j, q_{j+1}, x, y\}$ forming tetrahedra with good edges. Moreover each $q_j$ is at most distance 2 from the from the path $\{p_1, p, p_2\}$ in $CV$, where $V$ is (the possibly disconnected surface) $\Sigma \setminus (x, y)$.
\end{sublemma}
\begin{proof}    
    Suppose $\{p, x\}$ is the bad edge. Let $L, R$ be the two non-pants components of $\Sigma \setminus (p, x)$. Consider how the curves $p_1, p_2$ and $y$ distributed across these components. If they are all in one component, say $R$, then since $L$ is not a pants, it has a genus or at least two punctures. Thus in this case we could take the path $\{q_j\}$ to be this single curve.

    Suppose instead we had $y, p_2 \in CR$ and $p_1 \in CL$. If $\xi(L) \geq 3$, then we can use \autoref{findgoodedge} to find a good, new $q$ in $L$ disjoint from $p_{1}$. If $\xi(L) < 3$ then we must have $\xi(R) \geq 6$ so we can find a suitable $q$ in $CR$ via \autoref{findgoodtriangle}. Once again in this case, we can take the path $\{q_j\}$ to consist of the single curve $q$.

    Finally we have the case where $y \in CR$ and $p_1, p_2 \in CL$. If $R$ has enough complexity, in particular if $\xi(R) \geq 3$, then we can use \autoref{findgoodedge} to get a suitable curve as we have done above. Thus we are left with the case $y \in CR$ and $p_1, p_2 \in CL$ where $\xi(R) \leq 2$. In particular there is not enough complexity in $R$ so we will have to find suitable curves in $L$. Since it's possible that $p_1, p_2$ fill $L$, this is where we may need to find a sequence of good curves as opposed to a single good curve.

    First we need to ensure that $L$ has enough complexity to allow us to find a path of good curves. This is easy to see: we know that $\xi(\Sigma) \geq 10$ and $\xi(R) \leq 2$. Since $\xi(\Sigma) = \xi(L) + \xi(R) + 2$, we conclude that $\xi(L) \geq 6$.

    As we've done previously, we will use \cite{wright2024} to find a path of good curves with good edges (of course edges with other curves, such as $x$, may be bad but we will iteratively fix this). Since $\xi(L) \geq 6$, one of Lemmas 3.12, 3.13 and 3.15 holds, which gives a path $\{q_j\}$ of curves in $CL$ connecting $p_1$ and $p_2$. If all $\{q_j\}$ form good edges with $x$ then we are done (notice $y$ is necessarily a pants curve in this case so all non-separating curves form good edges with $y$). 
    
    So suppose there is some $\{q_j\}$ which forms a bad edge with $x$. In other words, $q_j$ is non-separating on $\Sigma$ but jointly separating with $x$. Look at $\Sigma \setminus(q_j, x)$ which has 2 non-pants components, say $L', R'$. Notice that $p, y$ must lie in the same component, say $R'$ (see \autoref{diamondsurface}).

    As we did previously, consider how $q_{j-1}$ and $q_{j+1}$ are distributed across $L', R'$. If they are both also in $R'$ then since $L'$ is not a pants, we can find a good curve in $L'$. This curve necessarily forms good edges with $q_{j-1}, q_{j+1}, x, y$ and so we can replace $q_j$ with this.

    Once again it might happen that $q_{j-1}$ and $q_{j+1}$ are actually in different components, say $q_{j-1} \in CL'$ and $q_{j+1} \in CR'$. Now will cut out $p$ as well: let $S, T, W$ be the 3 components of $\Sigma \setminus (q_j, x, p)$ (one of these components is simply $L'$). One of the components must have complexity greater than 2; if not then we would have 
    $$\xi(\Sigma) = \xi(S) + \xi(T) + \xi(W) + 3 \leq 9$$
    Moreover each component has exactly 2 boundaries which allows us to find a replacement for $q_j$ by \autoref{findgoodedge}.

    It might instead happen that $q_{j-1}, q_{j+1} \in CL'$. As before if $\xi(R') \geq 3$ then we can find a suitable replacement for $q_j$ in $CR$. Thus the only problem that may occur is if $\xi(R') \leq 2$. In this case we will repeat what we did above one final time.

    Let $\{r_k\}$ be a path in $CL'$ connecting $q_{j-1}$ and $q_{j+1}$. Once again we are done unless there is an $r_k$ which forms a bad edge with $x$ in such a way that $\Sigma \setminus (x, r_k)$ has two non-pants components $A, B$ with $r_{k-1}, r_{k+1} \in CA$ and $y \in CB$. However in this case we are guaranteed that $B$ has enough complexity to have a good curve disjoint from $y$ since $R \subset R' \subset B$ which implies $\xi(R) < \xi(R') < \xi(B)$. Since $\xi(R) \geq 1$ (recall that $R$ is non-pants by assumption) we are done. See \autoref{diamondsurface} for reference.
\end{proof}

\begin{figure}[h]
\centering
\includegraphics[width=.7\linewidth]{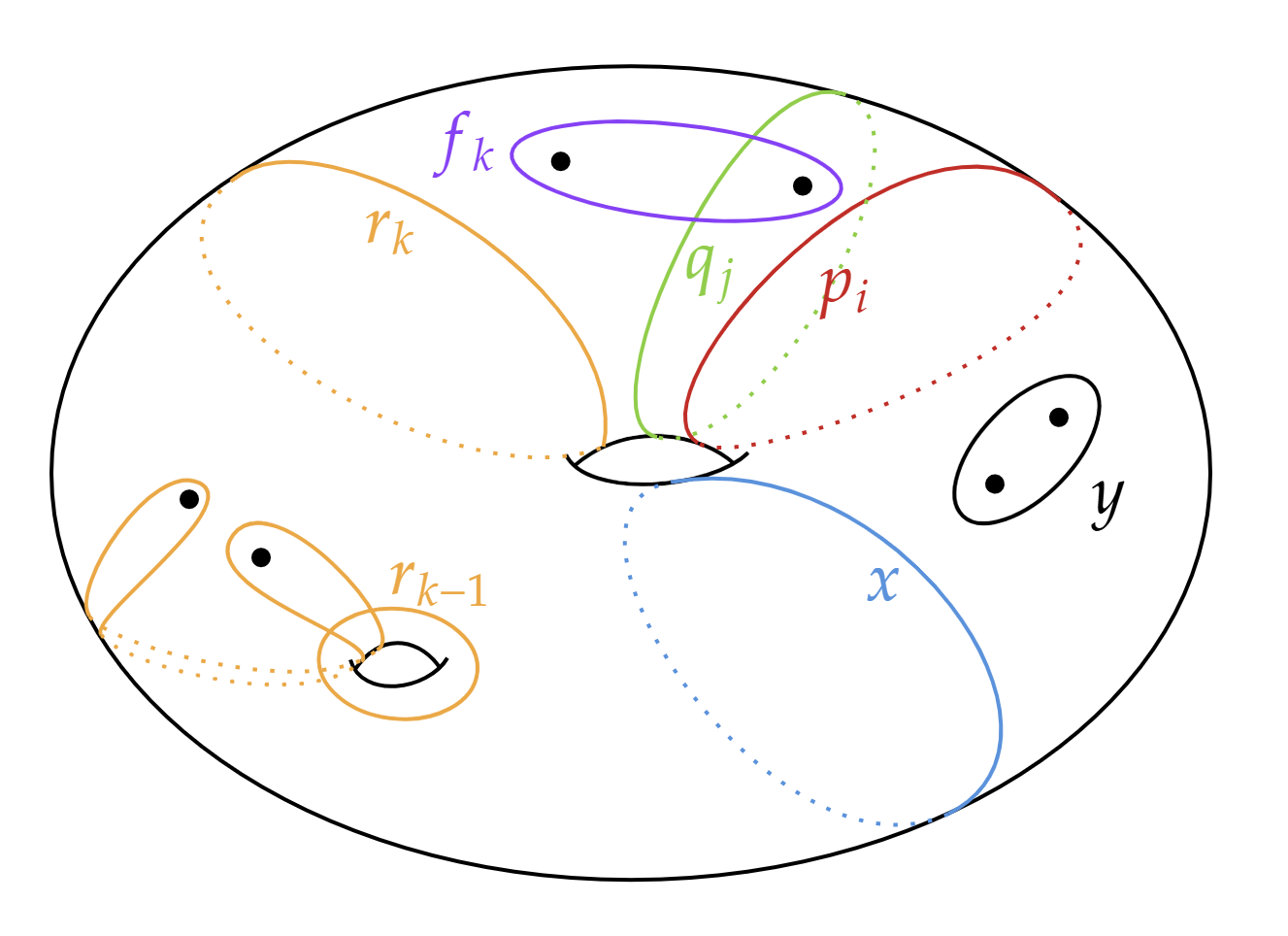}
\caption{The curve $f_k$ is good and guaranteed to exist by complexity considerations}
\label{diamondsurface}
\end{figure}

\begin{figure}[h]
\centering
\includegraphics[width=.7\linewidth]{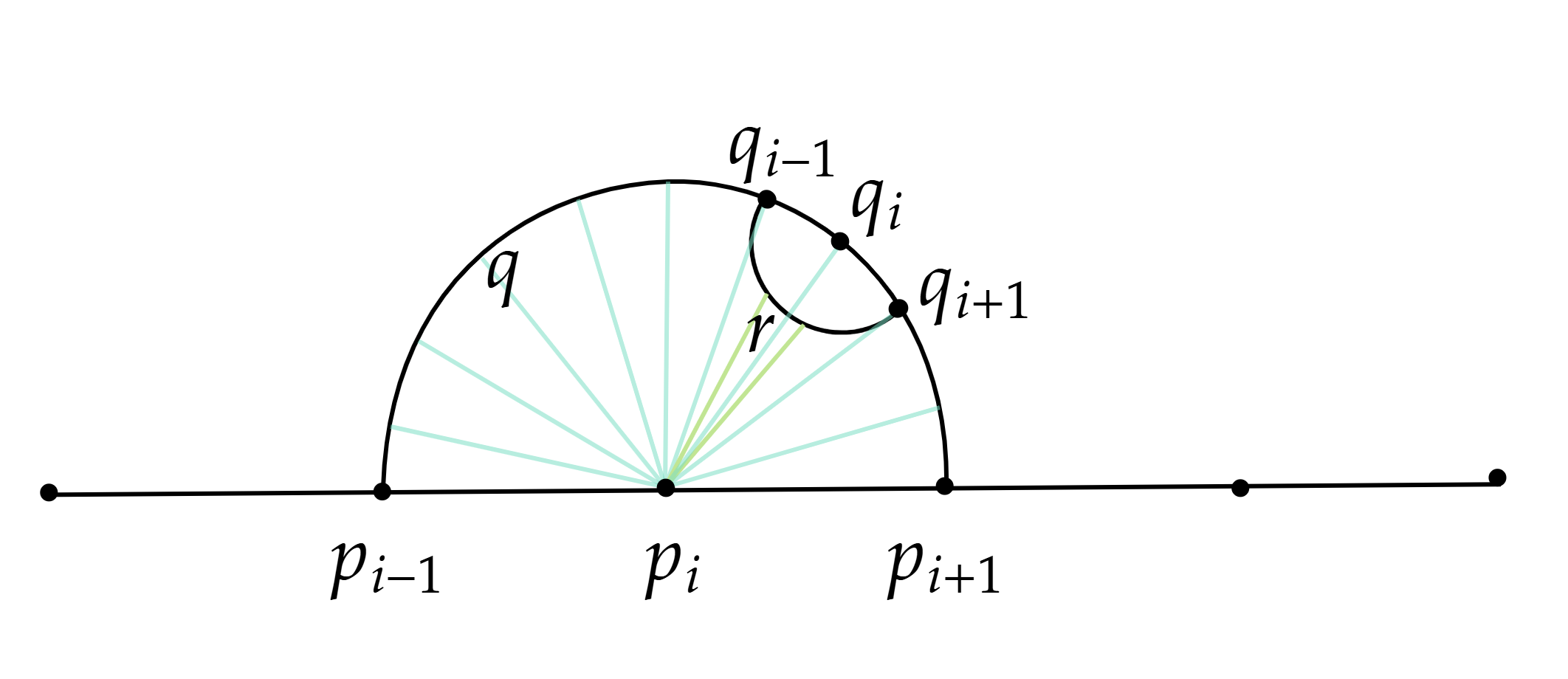}
\caption{connecting by paths in subsurfaces of $V$}
\label{diamond}
\end{figure}

Now we finally put everything together to prove the main theorem of this section.
\begin{proof}[Proof of \autoref{inductionlem2}]
Suppose $(a,b,x,y)$ is a good diamond. Then by \autoref{connecttoO1} there are paths $a=p_1,p_2,...p_{M+C}$ and $b=q_0,q_1,...q_{M+C}$ that form tetrahedra with good triangles with $x,y$. Then by \autoref{Ozisconnected1}, there is a path $p_{M+C}=r_0,r_1,...r_k=q_{M+C}$ that forms tetrahedra with good edges with $x,y$.

Suppose $r_i,r_{i+1},x,y$ is a tetradehron with good edges but not good triangles. We will find some $w_i\in S_{r+1}$ disjoint from $r_i,r_{i+1},x,y$ so that $(r_i,w,x,y)$, $(r_{i+1},w,x,y)$ are both tetrahedra with good triangles.

WLOG suppose $r_i,x,y$ is a bad triangle. Then $\Sigma\setminus (r_i,x,y)$ has two components $L,R$. Say $r_{i+1}\in CR$. Since $\xi(\Sigma)\geq 7$, either $\xi(L)\geq 2$ or $\xi(R)\geq 4$. If $\xi(L)\geq 2$: there is a genus or two punctures on $L$, in which case take $w$ to be a curve around the genus or the two punctures. If $\xi(R)\geq 4$, there is either a genus or two punctures on $R\setminus r_{i+1}$, so take $w$ to be a curve around the genus or the two punctures. 

Then the concatenation
$$ a=p_1,p_2,...p_{M+C},r_0,w_0,r_1,w_1,...r_{k-1},q_{M+C},...q_1,q_0=b$$
is the desired path, where $w_i=r_i$ if $(r_i,r_{i+1},x,y)$ is already a tetrahedron with good triangles.
\end{proof}

We also require this following proposition for later to show some things about the simple connectivity of the subcomplex spanned by the 1-skeleton of $(\Upsilon,\Sigma)$ in $C\Upsilon$.

\begin{proposition}\label{fixdiam}
    Fix $\xi(\Sigma)\geq 10$. Suppose $\Upsilon=\Sigma\setminus z$ where $z$ is good in $\Sigma$ and suppose $x,y,a,b\in C\Upsilon$ are vertices that are good in $\Sigma$ of a diamond. Also suppose that all edges besides $(x,y)$ of the diamond are good in $\Sigma$. Then there is a path $a=p_0,\dots,p_n=b$ where $(p_i,p_{i+1}),(p_i,x),(p_i,y)$ are edges good in $\Sigma$. 
\end{proposition}

\begin{proof}
    The idea is that we try to find some $p$ such that it makes good edges in $\Sigma$ with all of $x,y,a,b$. When this is not possible, we construct a path.

    Let $L,R$ be the two non-pants components of $\Upsilon\setminus\{x,y\}$ and $R\subset R'$ where $L,R'$ are the two non-pants components of $\Sigma\setminus \{x,y\}$. This is well defined since the edge $(x,y)$ is not good in $\Sigma$ which implies it is not good in $\Upsilon$ as well.

We will use some sublemmas:
\begin{sublemma} \label{tube}
    Let $\Upsilon$ be a non-pants component of $\Sigma\setminus\{x,y,z\}$ where $x,y$ are good vertices but not a good edge in $\Sigma$, and $(x,z)$, $(y,z)$ are good edges. If $\Upsilon$ has 2 or 4 boundary components, and $\xi(\Upsilon)\geq 1$, then there is some $a\in C\Upsilon$ that is good in $\Sigma$ and makes good edges with any good curve $b$ which does not cut $\Upsilon$. If $\Upsilon$ has 3 boundary components and $\xi(\Upsilon)\geq 2$ then the same conclusion applies.
\end{sublemma}
\begin{proof} \renewcommand{\qedsymbol}{$\blacksquare$}
    If $\Upsilon$ has two boundary components then there are either two punctures or a genus on $\Upsilon$ so there is some $a$ which forms good edges any curve missing $\Upsilon$.

    If $\Upsilon$ has 4 boundary components then 2 of them come from $z$ and $z$ is non-separating on $\Sigma$ so we take $a$ to be a non separating curve on $\Sigma$ going between the genus of $z$ and that of $x$ and $y$.

    If $\Upsilon$ has 3 boundary components, then $z$ is a pants curve and by the complexity count there are 2 punctures on $\Upsilon$ so we take $a$ to be a pants curve.
\end{proof}

\begin{sublemma}\label{tube2}
    Let $\Upsilon$ be a non-pants component of $\Sigma\setminus\{x,y\}$ where $x,y$ are good vertices but not a good edge in $\Sigma$. If $\xi(\Upsilon)\geq 3$ then for any $p\in C\Upsilon$ which is good in $\Sigma$ and $p,y$ is a good edge, there is some $a\in C\Upsilon$ that is good in $\Sigma$ and makes good edges with $p$ and any good curve $b$ which does not cut $\Upsilon$.
\end{sublemma}

\begin{proof} \renewcommand{\qedsymbol}{$\blacksquare$}
    We know that $x$ and $y$ share a genus. Suppose $p$ is non separating in $\Upsilon$. Then we may take $a$ to be a curve that goes from the genus of $p$ to the one of $x$. Then $a,p$ is jointly non-separating. Now take any $b$ which misses $\Upsilon$ ($b$ may be $x$ or $y$). If $b$ is separating on $\Sigma$ then it must be a pants curve, in which case $a,b$ is a good edge. If $b$ is not separating on $\Sigma$ then $\Sigma\setminus \{a,b\}$ is connected so it is a good edge.

    Now suppose $p$ is separating on $\Upsilon$ but not a pants curve. Then since $p,y$ is a good edge, it must be that $p$ cuts out a pants on $\Upsilon$. So then $\xi(\Upsilon\setminus p)=\xi(\Upsilon)-1\geq 2$. Then since $\Upsilon\setminus p$ has two boundary components, this means it has at least 3 punctures or genus so there is some good $a$ which forms good edges with $p$ and any curve missing $\Upsilon$.

    Now suppose $p$ is a pants curve. Then $\Upsilon\setminus p$ has 3 boundary components, so there is at least 2 punctures or genus on $\Upsilon\setminus p$ so again there is some good $a$ which forms good edges with $p$ and any curve missing $\Upsilon$.
\end{proof}

    Suppose both $a,b\in CR$. Then since $L$ is not a pants and has two boundary components, there's some $p\in CL$ such that $p$ makes good edges in $\Sigma$ with all of $x,y,a,b$ by \autoref{tube}.

    Suppose $a\in CL,b\in CR$. Then there is some $p$ either in $CL$ or $CR$ that makes good edges in $\Sigma$ with all of $x,y,a,b$, since $\xi(L)\geq 3$ by \autoref{tube2} or $\xi(R)\geq 4$, which is satisfied since $\xi(L)+\xi(R)\geq 6$.

    Suppose $a,b\in CL$. If $\xi(R)\geq 2$ then there is some $p\in CR$ that makes good edges in $\Sigma$ with $x,y,a,b$.

    Now suppose $a,b\in CL$ and $\xi(R)\leq 1$, which means $\xi(L)\geq 6$, so by \cite{wright2024} we have a path $a=p_0,...p_k=b$, $p_i\in CL$, $(p_i,p_{i+1})$ a good edge in $\Sigma$. 

    Now we need to arrange it so that all $p_i$ make good edges with $x,y$. This is the same procedure as \autoref{Ozisconnected1}. Say $p_i,x$ is a bad edge in $\Sigma$. Look at the two components of $L\setminus p_i$, which are $T$ and $B$. If $p_{i-1}\in CT,p_{i+1}\in CB$ are in different components, one of $\xi(T\setminus p_{i-1}), ~\xi(B\setminus p_{i+1})\geq 2$, or else $\xi(L)=\xi(T)+\xi(B)+1\leq 2+2+1=5$ and then $\xi(\Sigma)\leq 5+1+2+1=9$. But then there is some $p_i'$ in $CT$ or $CB$ that makes good edges with $x,y,p_{i-1},p_{i+1}$. 

    Similarly, if $p_{i-1},p_{i+1}\in CB$ and $\xi(T)\geq 1$ we could find such $p_i'$.
    
\begin{figure}[h]
\centering
\includegraphics[width=1\linewidth]{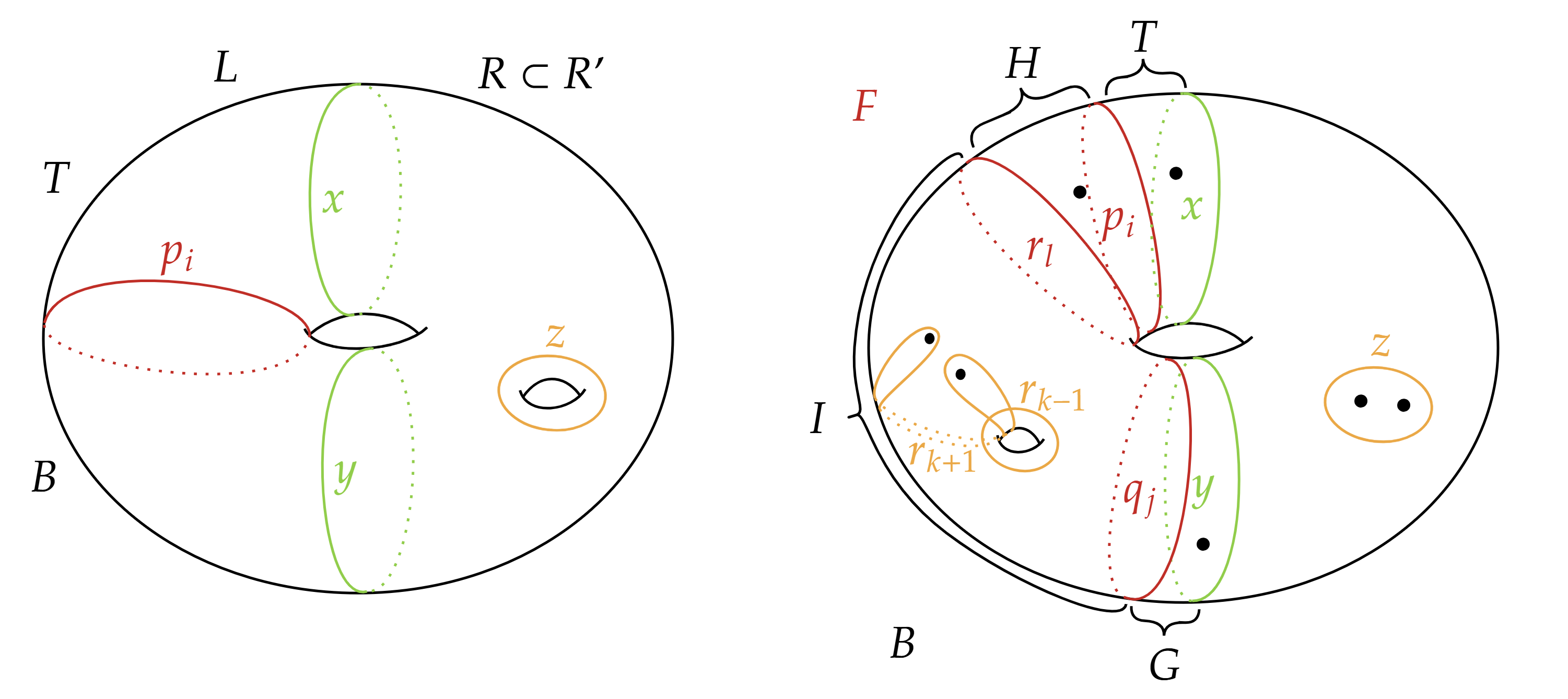}
\caption{Choosing a curve that makes good edges with $x,y,z$ and neighbors in the path}
\label{tubetube}
\end{figure}

    So suppose $\xi(T)=0$ and $p_{i-1},p_{i+1}\in CB$ and so $\xi(B)=\xi(L)-\xi(T)-1\geq 5$. Then we may find by \cite{wright2024} a path
    $$p_{i-1}=q_0,...q_l=p_{i+1}$$
    where $q_j,q_{j+1}$ are good edges in $\Sigma$. But now say $q_j,x$ is not a good edge in $\Sigma$ (in fact if $q_j,y$ is not a good edge then that implies $q_j,x$ is also not a good edge so we only need to consider this case). Then $q_j$ separates $B$ into components $F,G$, and if $q_{j-1},q_{j+1}$ are in different components, we have $$\xi(F)+\xi(G)+1=\xi(B)\geq 5$$, so either one of $\xi(F)$ or $\xi(G)\geq 3$, or both are at least 2. In any case, $B\setminus \{q_{j-1},q_{j+1}\}$ has 2 punctures so we replace $q_j$ with a pants curve. 

    If $q_{j-1},q_{j+1}$ are both in $G$ then we replace $q_j$ with a non-separating or pants curve using a puncture or genus from $F$ and from between $p_i$ and $x$.
    
    So suppose $q_{j-1},q_{j+1}$ are both in $F$. If $\xi(G)\geq1$ then there's some curve $q'$ that is either a pants or non-separating and makes good edges with all the other curves. So suppose not. Then $\xi(F)\geq 4$ and we may find a path
    $$q_{j-1}=r_0,...,r_u=q_{j+1}$$
    in $F$ where the path has good edges in $\Sigma$. 

    Now suppose there is some $r_l$ where $r_l,x$ is a bad edge. Let $H,I$ be the two components of $F\setminus r_l$. If $r_{l-1},r_{l+1}$ are in different components of $H,I$, say $r_{l-1}\in CH$, and we have that one of $\xi(H),\xi(I)$ is at least 2, WLOG say $\xi(H)\geq 2$, because $\xi(H)+\xi(L)+1=\xi(F)\geq 4$. Then there is at least a puncture or genus in $H\setminus r_{l-1}$. Together with the puncture between $p_i$ and $x$, we may replace $r_l$ with a pants curve.

    Now suppose $r_{l-1},r_{l+1}$ are in the same component, suppose they are in $I$. See \autoref{tubetube}. Then $H$ has at least a puncture or genus, so together with the puncture between $p_i$ and $x$, replace $r_l$ with a pants curve that is disjoint from $x,y,r_{l-1},r_{l+1},z$. If they were both in $H$ we would take a pants curve around punctures in $I$ and between $q_j$ and $y$.

    This gives us a good path from $a$ to $b$ where all the edges with $x,y$ are good.
    \end{proof}

\subsection{Cone}\label{cone}

In this section we deal with the cone, which is where there is a vertex $z$ in our triangulation in $S_r$, and the link of this vertex is in $S_{r+1}$. We will replace $z$ and find another triangulation in $S_{r+1}$ of the link of $z$.

\begin{definition}
Fix $z\in C\Sigma, z\in S_r$. Call a loop $X=\{x_0, x_1,...x_n=x_0\}\subset S_{r+1}$ a good \emph{cone} (a cone with good edges) with respect to $z$ if every tuple $(x_i,x_{i+1},z)$ is a good triangle. Call a cone $Y=\{y_0, \dots, y_m=y_0\}$ \emph{homotopic to $X$} if $Y$ is a cone with respect to $z$ and the loop $y_0, \dots, y_m=y_0$ is homotopic to $x_0, x_1, \dots, x_n=x_0$ through the curve complex.
\end{definition}

A useful fact:
\begin{lemma}\label{factedge}
    If $x,y,z$ is a good triangle in $\Sigma$ and $U=\Sigma\setminus z$, then $x,y$ is a good edge in $U$.
\end{lemma}
\begin{proof}
    We check that $x$ is good in $U$. Suppose not: $x$ separates $U$ into two non-pants. But this means $\Sigma\setminus \{x,z\}$ has two non-pants so is not a good edge in $\Sigma$.

    Now suppose $x,y$ is a bad edge in $U$. Then $U\setminus \{x,y\}=\Sigma\setminus\{x,y,z\}$ has two non-pants so $x,y,z$ is a bad triangle.
\end{proof}

\begin{lemma}[Cone lemma]\label{lem:induction3}
Let $z\in S_r$, $\xi(\Sigma)\geq 8$, and $U$ be the unique non-pants component of $\Sigma\setminus z$. Fix any $M\in \NN$.
Then, given a good cone $X=\{x_i\}_{i=0}^n$, there is a family of good homotopic cones $Y_i=\{y_{i,0}, ,...y_{i,m}=y_{i,0}\}\subset S_{r+1}, i=0,...M$ such that $d_U(y_{M,j}, o)\geq M$ for all $j$. Moreover the homotopy contains only good triangles.
\end{lemma}
\begin{proof}
Fix a good cone $X=\{x_i\}$. We will construct a homotopic cone $Y$ such that $d_U(Y,c)=d_U(X,c)+1$ where $d_U(Y,c)=\min_{y\in Y} \{d_U(y,c)\}$. The result follows by iterating this $M$ times.

Let $V$ be the unique component of $S\setminus (x_j\cup x_{j+1}\cup z)$ that isn't a pants, and $U$ be the unique component of $S\setminus z$ that isn't a pants.

For each $j$ we first construct some $a_j$ that is disjoint to $x_j,x_{j+1}$, see \autoref{type1}. For each $j$, since $x_j,x_{j+1},z$ is a good triangle, by \autoref{inductionlem1}, there is a curve $a\in S_{r+1}$ that forms a good tetrahedron $a,x_j,x_{j+1},z$. 
From \autoref{inductionlem1} we also get that $d_V(a,c)>N$, where $N$ is the constant from the BGIT. Then, since $V\subset U$, this means 
any geodesic from $a$ to $c$ in $CU$ must miss $V\subset U$. But $x_j,x_{j+1}$ are the only such curves since they are a good edge in $U$ by \autoref{factedge}, so $d_U(a,c)= \min\{d_U(x_j,c), d_U(x_{j+1},c)\}+1$. 
Hence for each pair $x_j,x_{j+1}$, we get a corresponding $a_j$. Since there are $n$ such pairs we get $n$ points $a_0,a_1...a_n=a_0$. 

\begin{figure}[h]
\centering
\includegraphics[width=.5\linewidth]{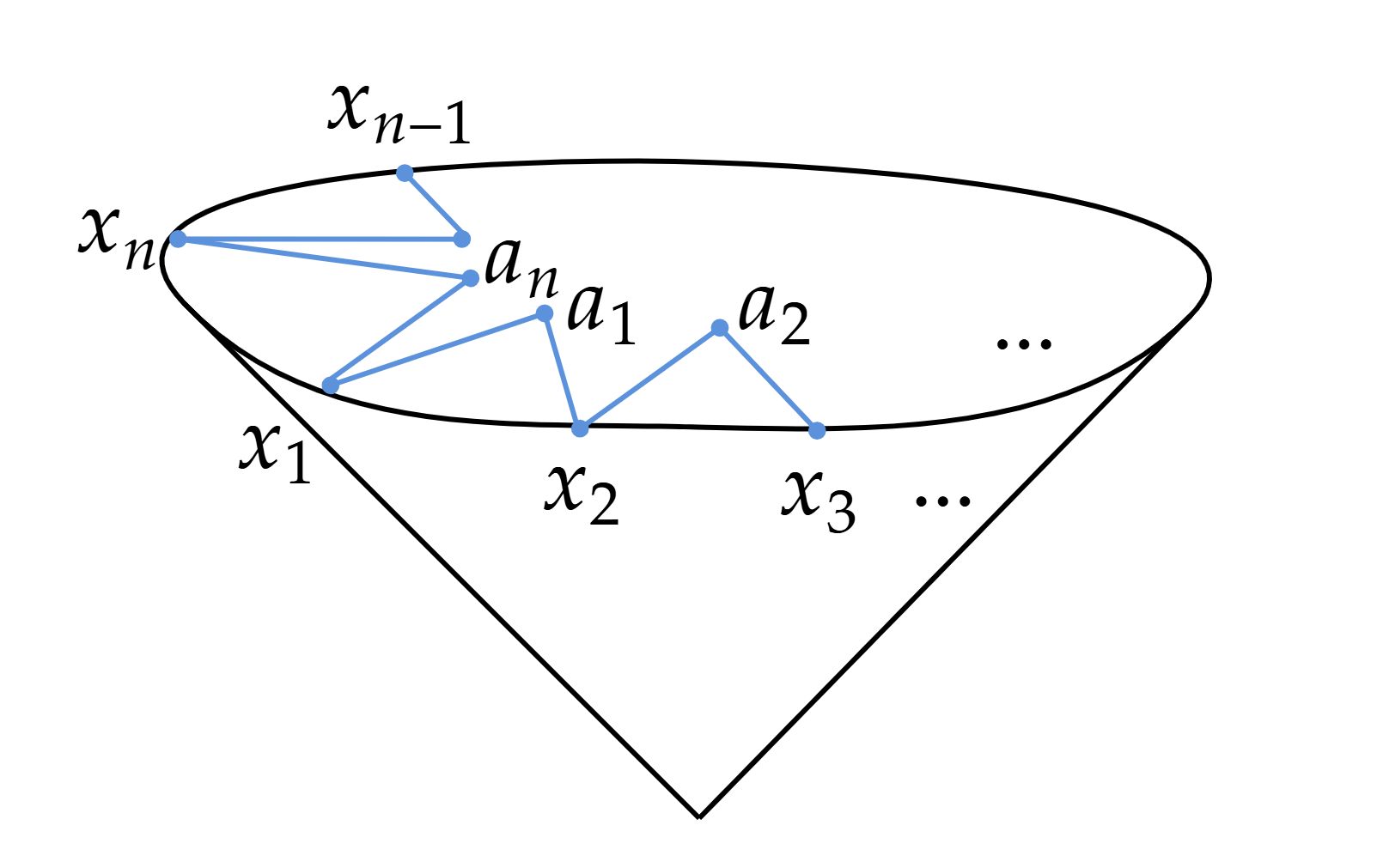}
\caption{Constructing a homotopic cone}
\label{type1}
\end{figure}   

Now we connect these $a_i$ to form a homotopic cone. We have that $a_{i-1},a_i,x_i,z$ is a good diamond, so by \autoref{inductionlem2} there is a path $b_0=a_{i-1},b_1,...b_m=a_i$ where $b_j,b_{j+1},x_i,z$ are tetrahedra with good triangles. By \autoref{inductionlem2} we get that all these $b_j$ are such that $d_W(b_j,c)>N$ where $W$ is the unique non-pants of $\Sigma \setminus (x_i\cup z)$, so then any geodesic from $b_j$ to $c$ in $CU$ must miss $W\subset U$, so must pass through $x_i$ since $x_i$ is good in $U$. So $d_U(b_j,c)= d_U(x_i,c)+1$. 
Now for each $i$ we get a path $a_{i-1}=b_{i,0},b_{i,1},... b_{i,m_i}=a_i$, so concatenating all these paths gives us a cone $B=\{a_0=b_{1,0},b_{1,1},...b_{n,0},...b_{n,m_n}=b_0\}$ such that $d_U(B,c)= 
d_U(X,c)+1$. Also, by construction $X$ is homotopic to $B$, where all the triangles in the homotopy are good.


\end{proof}

\begin{lemma}\label{getgoodlink}
    Let $\xi(\Sigma)\geq 10$, $o\in C\Sigma$ is good and $U=\Sigma\setminus o$. Let $X$ a loop in $CU$ with good edges in $\Sigma$. Then any triangulation $\Delta$ of $X$ is homotopic (relative to $X$) in a $2$-neighborhood of $\Delta$ to a triangulation $\Delta'$ of $X$ where the link of any bad vertex has only good vertices in $\Sigma$.
\end{lemma}
\begin{proof}
    We will need to first get rid of a particular set of bad vertices. Let $W(\Delta)=\{w\in (U,U) \cap \Delta \mid w\notin (U,\Sigma)\}$ (these are vertices that appear good in $U$ but are not good in $\Sigma$). Pick some $w\in W(\Delta)$. 
    
    The first observation to make is that the link of $w \in W(\Delta)$ cannot contain other vertices in $W(\Delta)$. This is because if $w\in W(\Delta)$ then $w$ must be a curve that cuts out a pants on $U$ and for $w$ to be bad in $\Sigma$ it must be that $w$ either cuts out a four-punctured sphere and a non-pants on $\Sigma$ or a punctured torus and a non-pants on $\Sigma$. In either case, if $x$ were disjoint to $w$ then $x$ must lie in the unique non-pants of $U\setminus w$, in which case it is either good in $\Sigma$ or bad in both $U$ and $\Sigma$.
    
    Now we would like to replace $w$ with some triangulation of the link of $w$. So let $V=U\setminus w$ and we have that $\text{link}(w)\subset CV$. Since $\xi(V)\geq 8$, let $\Delta_w\subset CV$ be a triangulation of $\text{link}(w)$. We have that $\Delta_w$ is disjoint from $w$ so the triangulation $\Delta_1$ we obtain from replacing each $w\in W(\Delta)$ with a triangulation $\Delta_w$ is homotopic in a 1-neighborhood to $\Delta$ and $W(\Delta_1)$ is empty.
    
    Now let $z\in \Delta_1$ which is not in $(U,\Sigma)$. 
    The link of $z$ is $y_0,y_1,...,y_k=y_0$ 
    and say $y_i$ is also not in $(U,\Sigma)$. Look at the diamond $y_{i-1},y_i,y_{i+1},z$.
    
    Since $z,y_i$ are both not in $(U,\Sigma)$ or $(U,U)$, 
    $U\setminus(z \cup y_i)$ has 3 components say $A, B, C$ and at most one of the 3 components can be a pants. 
    If there is a non-pants component which does not contain either of $y_{i - 1}$ or $y_{i + 1}$, we can pick $w$ to be a vertex good in $\Sigma$ in this component. 
    
    
    If that does not happen, it must be that up to relabeling, $B$ is a pants and $y_{i-1} \in A$ and $y_{i+1} \in C$.
    In this case, due to the complexity assumption of $\xi(\Sigma)\geq 10$, one of $A,C$ must have complexity at least $4$, say $A$ does. Then there is a $w$ in $CA$ disjoint from $y_{i-1}$, where $w$ is good in $\Sigma$. To see this look at $A\setminus y_{i-1}$, which is either one or two components. If it were one component, $\xi(A\setminus y_{i-1})\geq 3$, in which case we can find a good curve by the following argument.
    
    If $A \setminus y_{i-1}$ has at most 4 boundary components, then by complexity counting we see that this subsurface must have a genus or at least a pair of punctures. If $A \setminus y_{i-1}$ has more than 4 boundary components then it must be that $o$ and $y_{i-1}$ were both non-separating. In this case we can take the good curve to be a curve going between the genera of $o$ and $y_{i-1}$. To be precise, let $\gamma$ be an arc from $o$ to $y_{i-1}$. Then the curve going between the genera is given by the concatenation $\gamma*o* \gamma^{-1} * y_{i-1}$.
    
    If $A \setminus y_{i-1}$ were two components then by definition $y_{i-1}$ is a separating curve (on $A$). So if $A$ had a genus then the genus would be located on one of these components and we take a non-separating curve around it. If $A$ had no genus then it must have had at least 4 punctures and so there must be at least 2 punctures on one of these components and we take a pants curve around them.
    
    So there's a triangulation $\Delta_2$ homotopic to $\Delta$ in a 1-neighborhood of $\Delta_1$ where the link of each vertex consists of only good vertices in $\Sigma$.
    \end{proof}

    \begin{lemma}\label{linkinonecomponent}
        Let $\xi(\Sigma)\geq 10$, $o\in C\Sigma$ is good and $U=\Sigma\setminus o$. Let $X$ a loop in $CU$ with good edges in $\Sigma$. Let $\Delta$ be a triangulation of $X$ with $W(\Delta)=\emptyset$ (see \autoref{getgoodlink}) such that the link of every bad vertex $z$ has only good vertices. Then there exists a triangulation $\Delta'$ in the 1-neighborhood of $\Delta$ such that the link of every bad vertex $z$ has only good vertices and in addition also lies in the same component of $\Sigma \setminus z$.
    \end{lemma}
    \begin{proof}
    Due to the complexity assumption, $\xi(\Sigma)\geq 7$, so $z$ splits $\Sigma$ into two components $L,R$ with $o\in CR$. We first consider the case $\xi(L) \geq 3$. Label the link of $z$ as $y_0,...y_{n_z}=y_0$. Let $R(z,\Delta)$ be the number of vertices in the link of $z$ with respect to a triangulation $D$ that lie in $R$.

    Let $\Delta'$ be a triangulation in the 1-neighborhood of $\Delta$ such that $\sum_{z\in\chi(\Delta')} R(z,\Delta')$ is minimal. We will show if this quantity is positive there exists another triangulation for which it is strictly smaller.
    
    Say $z$ is such that $R(z,\Delta')$ is positive and $y_i \in CR$. If either of $y_{i-1},y_{i+1}\in CR$ then by the complexity assumption on $L$ there is a $w\in CL$ disjoint from whichever one of $y_{i-1},y_{i+1}$ is in $L$ and good, by Wright lemma 3.10. As usual this is done by complexity counting and observing that $L$ has exactly one boundary component, the one from $z$. So $L$ must have (at least) two genus, one genus and at least 2 punctures or no genus and at least 5 punctures.
    
    If both $y_{i-1},y_{i+1}$ are in $L$, we note that because of the complexity of $L$, we can use Wright lemmas 3.12-3.15 to get that there is a path $y_{i-1}=w_0,w_1,...w_l=y_{i+1}$ such that these $w_i\in CL$. Since these $w_i$ form tetrahedra $w_i,w_{i+1},z,y$, this gives a homotopy from the diamond $z,y_{i-1},y,y_{i+1}$ to this sub triangulation $z,y_{i-1},y,y_{i+1}, \{w_j\}_j$. So we can replace $\Delta'$ in either case with some $\tilde \Delta$ such that
    $$\sum_{z\in\chi(\tilde \Delta)} R(z,\tilde \Delta)<\sum_{z\in\chi(\Delta)} R(z,\Delta')$$
    But $\chi(\tilde \Delta)=\chi(\Delta)$ and $\Delta'$ was chosen to minimise the sum so we get a contradiction.

    The other case that may happen is that the complexity of $L$ is too small for the above argument. In this case we need to move the link to $R$ but we also need to be careful about avoiding $o$.

    So suppose now that $\xi(L) < 3$. In this case we must have $\xi(R) \geq 7$ so $\xi(S) \geq 6$ where $S = R \setminus o$. Suppose $y_i$ is a vertex in the link of $z$ that lies in $L$. The simplest case is if there exists a good $y_i' \in CS$ disjoint from $y_{i-1}$ and $y_{i+1}$. This occurs if the latter two curves both lie in $CL$ or even if exactly one of them does. However it may be that $y_{i-1}, y_{i+1}$ lie in $CS$ and are filling. In this case, as always, we appeal to \cite{wright2024} to find a path of good vertices with good edges connecting $y_{i-1}$ and $y_{i+1}$ in $CS$. If $S$ has at least 2 genus we can apply 3.12. Since $S$ has at most 3 boundary components, it has at least 3 punctures in the genus 1 case and at least 6 in the genus 0 case. Thus one of 3.13 or 3.15 of \cite{wright2024} hold to give us connectivity. If all of the vertices form good edges with $o$, we are done.

    Suppose there is some $y_i$ which forms a bad edge with $o$. Consider $R \setminus (o, y_i)$ which has 2 non-pants components say $R_1, R_2$. Since $R$ has complexity at least 7, we must have that one of $R_1, R_2$ has complexity at least 3. Say it is $R_1$. Since either component can have at most 3 boundary components, it must be that $R_1$ has either a genus or at least 3 punctures and hence we can find a good curve to add to the triangulation making good edges with $o, y_i, y_{i-1}$ and $y_{i+1}$.
    
    \end{proof}

    \begin{theorem}\label{Ccsimplycon}
    Let $\xi(\Sigma)\geq 10$, $o\in C\Sigma$ is good and $U=\Sigma\setminus o$. Let $X$ a loop in $CU$ with good edges in $\Sigma$. Then any triangulation $\Delta$ of $X$ is homotopic (relative $X$) in a $5$-neighborhood of $\Delta$ to a triangulation $\Delta^*$ of $X$ where $\Delta^*$ has good edges in $\Sigma$. In particular, the subcomplex generated by the 1-skeleton of $(U,\Sigma)$ in $CU$ is simply connected.
    \end{theorem}
    \begin{proof}
    
    Let $X=\{x_0,x_1,...x_n=x_0\}\subset (U,\Sigma)$ be a loop. Since $CU$ is simply connected, let $\Delta \subset CU$ be a triangulation of $X$. Let $V(\Delta)$ be the set of vertices in $\Delta$ that are not good and $E(\Delta)$ be the set of edges that are not good. We will find some triangulation $\Delta_5$ homotopic in the $5$-neighborhood of $\Delta$ such that $\#V(\Delta_5)+\#E(\Delta_5)=0$, which would mean $\Delta_5 \subset (U,\Sigma)$.
    
    Suppose that $V(\Delta)$ is not empty. Let $\chi(\Delta)$ be the set of bad vertices in $\Delta$ and $\chi(\Delta,z)$ be the number of bad vertices in $\Delta$ that are in the link of $z$. We use \autoref{getgoodlink} to ensure that all bad vertices are `isolated' in the sense that no bad vertices have neighbours which are also bad. To be precise, we claim that $\Delta$ is homotopic in its $2$-neighborhood in $C\Sigma$ to a triangulation $\Delta_2$ such that $\sum_{z\in \chi(\Delta_2)} \chi(\Delta_2,z)=0$.
    
    Next we can use \autoref{linkinonecomponent} to find $\Delta_3$ in the 1-neighborhood of $\Delta_2$ (hence in the 3-neighborhood of $\Delta$) such that not only does the link of every bad vertex $z$ contain only good curves but in addition, all the curves in the link lie in the same component of $\Sigma\setminus z$.

    Now, since $\Delta_3$ is such that the link of every bad vertex $z$ has only good curves that lie in the same component of $\Sigma \setminus z$, $z$ can be replaced with a good curve in the other component. Using the notation from \autoref{linkinonecomponent}, such a curve is easy to find if the link is contained in $R$, the component with $o$. If the link is contained in $L$, then the existence of a good curve in $R$ disjoint from $o$ requires us to use the fact that $z\notin W(\Delta)$, so $R\setminus o$ is not a pants.
    
    
    Now we have that $V(\Delta_3)=0$. Suppose $E(\Delta_3)>0$. So pick a bad edge $(x,y)$. Since it's not a boundary edge, the edge is the common edge of two triangles, say $(a,x,y)$ and $(b,x,y)$. Let $Q(x,y,\Delta')$ be the number of edges in $\{(a,x),(a,y),(b,x),(b,y)\}$ that are bad in a triangulation $\Delta'$. We claim that there's some $\Delta_4$ homotopic in the $1$-neighborhood of $\Delta_3$ such that $\sum_{(x,y) ~\text{bad}}Q(x,y,\Delta_4)=0$.
    
    \begin{figure}[h]
    \centering
    \includegraphics[width=.5\linewidth]{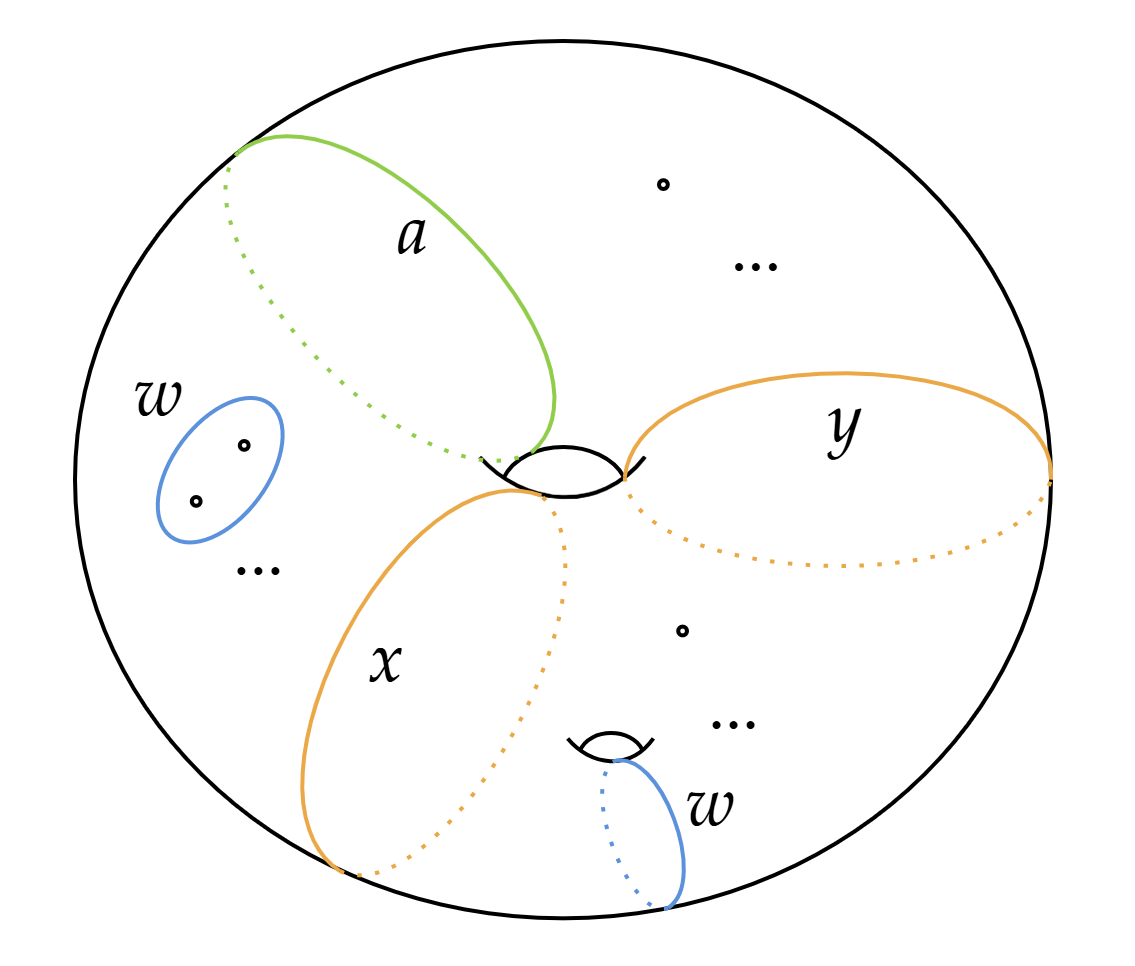}
    \caption{Bad edges in a triangle}\label{badedge}
    \end{figure}   
    
    Pick a bad edge $(x,y)$ with $(a,x)$ also a bad edge. Since these edges are both bad, up to homeomorphism they are the curves in \autoref{badedge}. This is because if $x,a$ are good vertices, and if $x$ were a pants curve, then $(x,a)$ would be a good edge. So the only instance of these bad edges is when $x,a,y$ are non-separating.
    
    We will find a $w$ so that $(w,a), (w,x), (w,y)$ are all good edges. Notice $S\setminus(a\cup x\cup y)$ has 3 components. Since $\xi(\Sigma)\geq 7$, there is a component $V$ with $\xi(V)\geq 2$ (if all the components were complexity 1, we would get $\xi(\Sigma)=6$). Thus $V$ either has at least 3 punctures or at least one genus. In the first case we can take $w$ to be a pants curve around 2 of the punctures and in the latter we can take $w$ to be a non-separating curve.
    
    So then we may homotope $\Delta_3$ in its $1$-neighborhood to some $\Delta_3'$ where
    
    $$\sum_{(x,y) \text{ bad in }\Delta_3'}Q(x,y,\Delta_3') ~< \sum_{(x,y) \text{ bad in } \Delta_3}Q(x,y,\Delta_3)$$
    
    But then since $\{(x,y) ~bad~in~\Delta_3'\}=\{(x,y) ~bad~in~\Delta_3\}$ and these sets are finite, we may apply this process a finite number of times to get a homotopic triangulation $\Delta_4$ in the $1$-neighborhood of $\Delta_3$ (and in the $4$-neighborhood of $\Delta$).
    
    So $\Delta_4$ is such that every triple $a,x,y$ has at most one bad edge. So let $(x,y)$ be a bad edge, and $(a,x,y),(b,x,y)$ each have exactly one bad edge (which they are sharing). By \autoref{fixdiam} we have a homotopy in the $1$-neighborhood of $\Delta_4$ to a triangulation $\Delta_5$ with $E(\Delta_5)=0$. Since by the step with $\Delta_3$ we have only good vertices in our triangulation and every step after we have only introduced good vertices, $V(\Delta_5)=0$ and so $V(\Delta_5)+E(\Delta_5)=0$ and $\Delta_5$ is homotopic to $\Delta$ in the $5$-neighborhood of $\Delta$.
\end{proof}

We must redo the above theorem for loops in $C\Sigma$ instead of $CU$. This is much easier now that we are working in $C\Sigma$.

\begin{theorem}\label{homotopins}
    Suppose $\xi(\Sigma)\geq 10$ and $\Delta$ is a triangulation of a good loop $\gamma\subset C\Sigma$. Then there is some $\Delta'$ a triangulation with good edges of $\gamma$ which is homotopic rel $\gamma$ to $\Delta$ in $B_5(\Delta)$. 
\end{theorem}

\begin{proof}
    We break this into a few lemmas.

\begin{lemma}\label{getgoodlinkSigma}
    Suppose $\xi(\Sigma)\geq 10$ and $\Delta$ is a triangulation of a good loop $\gamma\subset C\Sigma$. Then there is some $\Delta'$ a triangulation of $\gamma$ such that the link of all bad vertices are good, which is homotopic rel $\gamma$ to $\Delta$ in $B_1(\Delta)$. 
\end{lemma}
\begin{proof}\renewcommand{\qedsymbol}{$\blacksquare$}
    As we did earlier, we want to isolate bad vertices, i.e. ensure that the link of a bad vertex has only good vertices. Let $w$ be a bad vertex in $\Delta$ and let $y_i$ be a bad vertex in $\text{link}(w)$. Since $w,y_i$ are both bad, $\Sigma \setminus \{w,y_i\}$ is 3 components, say $E,F,G$ where $F$ has two boundary components and $E,G$ each have one. By assumption that $w,y_i$ are bad, $E,G$ are both not pants. Now if $y_{i-1},y_{i+1}$ are distributed across these components in such a way that there is a component that contains a good curve disjoint to all of $w,y_{i-1},y_i,y_{i+1}$ then we are done. We show that the complexity assumption gives us such a curve. 

    If $y_{i-1},y_{i+1}$ both miss $E$ or both miss $G$, we can find such a curve since there is another component with exactly one boundary component which is not a pants.

    If $y_{i-1}\in CE,y_{i+1}\in CG$ and $\xi(F)\geq 1$, we are once again done since there is a good curve in $F$.

    If $y_{i-1}\in CE,y_{i+1}\in CG$ and $\xi(F)=0$, we have $10\leq \xi(\Sigma)=\xi(E)+\xi(G)+2$ so WLOG $\xi(E)\geq 4$. Suppose $E\setminus y_{i-1}$ is one component, which means $E \setminus y_{i-1}$ has complexity at least 3 and has 3 boundary components. This means there is a genus on it or at least 3 punctures and we are done.
    
    It may instead be that $E\setminus y_{i-1}$ has two components say $A,B$, with $4\leq \xi(E)=\xi(A)+\xi(B)+1$. Thus one of $A$ or $B$ must have complexity at least 2. This component has at most 2 boundaries so it must have a genus or at least 3 punctures allowing us to find a good curve.
\end{proof}

\begin{lemma}
    Suppose $\xi(\Sigma)\geq 10$ and $\Delta$ is a triangulation of a good loop $\gamma\subset C\Sigma$. Then there is some $\Delta'$ a triangulation with good vertices of $\gamma$ which is homotopic rel $\gamma$ to $\Delta$ in $B_1(\Delta)$. 
\end{lemma}

\begin{proof}\renewcommand{\qedsymbol}{$\blacksquare$}
We begin by using \autoref{getgoodlinkSigma} to find a homotopic triangulation $\Delta'$ so that every bad vertex is isolated and has only good vertices in its link. Now naturally we want to replace the bad vertices. We will do this by moving the link of bad vertices to a subsurface of $\Sigma$ and then replacing the bad vertex with a good curve in the complementary subsurface.

Let $w$ be a bad vertex in the triangulation $\Delta'$. This means $\Sigma \setminus w$ has 2 non-pants components which we call $L$ and $R$. Let $R$ be the component with higher complexity. Notice we must have $\xi(R) \geq 5$. Our goal is to ensure that $\text{link}(w)$ lies entirely in $R$. Let $y_i$ be a curve in the link which does not. In other words, suppose $y_i \in CL$. 

Now that all the bad vertices have good links, we move these links into a large enough subsurface. Recall $\Sigma \setminus w=\{L,R\}$ where WLOG $\xi(R)\geq 5$. Say $y_i\in \text{link}(w)$ is such that $y_i\in CL$. If there exists a curve in $CR$ that is disjoint to both $y_{i-1}$ and $y_{i+1}$, then we can add this to our triangulation. If there is no such curve then by the complexity assumption we may use 3.10 in \cite{wright2024} to produce a path in $CR$ from $y_{i-1}$ to $y_{i+1}$ of good vertices and we are done.

Now we have a triangulation where the link of every bad vertex $w$ is contained in a subsurface of $\Sigma$. Moreover this subsurface is such that its complement is non-trival (what we labeled $L$ above which is non-pants by assumption). Thus we can replace $w$ with a good curve in the complement.

We may do this procedure to all bad vertices and get a homotopic triangulation where the homotopy is in $B_1(\Delta)$.
    
\end{proof}

\begin{lemma}
    Suppose $\xi(\Sigma)\geq 10$ and $\Delta$ is a triangulation with good vertices of a good loop $\gamma\subset C\Sigma$. Then there is some $\Delta'$ a triangulation with good edges of $\gamma$ which is homotopic rel $\gamma$ to $D$ in $B_1(\Delta)$. 
\end{lemma}
\begin{proof}\renewcommand{\qedsymbol}{$\blacksquare$}
This version is easier than the cone one. Take a bad edge $x,y$. $\Sigma\setminus \{x,y\}$ is two components, $L,R$ which are neither pants. The edge must be a common edge to a pair of triangle, say $(a, x, y)$ and $(b, x, y)$. If $a,b$ are in the same component, pick a curve in the other component that is good and makes good edges with $a,b,x,y$.

Otherwise, WLOG $b\in CR, ~\xi(R)\geq 4$ and so on a component of $R\setminus b$ there is a good curve that makes good edges with $a,b,x,y$.
\end{proof}
\end{proof}

\subsection{Almost simply connectedness}\label{almostsimplyconnected}
In this section, we apply the inductive steps to prove almost simply connectedness of spheres.

We begin by recalling the definition of almost simply connected. 

\begin{definition}\label{asc}
    A subspace of a metric space $U\subset X$ is called $N$-almost simply connected if for any loop $\gamma\subset U $, there is some continuous $f:D^2\to X$ such that $f(S^1)=\gamma$ and $f(D^2)\subset B_N(U)$ the $N$-neighborhood of $U$. We say $U$ is almost simply connected if there is some $N$ such that it is $N$-almost simply connected.
\end{definition}

Here $M$ is the constant from BGIT.

\begin{corollary}\label{inductivestep!2}
Let $\Sigma$ be a surface with $\xi(\Sigma) \geq 8$. Suppose for every surface $\Upsilon$ with $\xi(\Upsilon) = \xi(\Sigma) - 1$ we have that $C \Upsilon \setminus B_r$ is $N$-almost simply connected for every $B_r$. 

Now let $z \in S_r \subset C\Sigma$ be a good curve and let $U = S \setminus z$. Then for any $X \subset S_{r+1}$ which is a good cone with respect to $z$, there is a triangulation $\Delta$ of $X$ with good edges. Moreover we have $\Delta \subset C U \subset C \Sigma$ and $\Delta \subset S_{r+1}$.
\end{corollary}

\begin{proof}
Let $X$ a cone as above. By \autoref{lem:induction3} there is a homotopic cone $Y\subset S_{r+1}$, and $d_U(Y,o)>M+N+5$ where the homotopy has only good triangles and lives in $S_{r+1}$. 
Now
$Y\subset CU$ and $\xi(U)=\xi(\Sigma)-1$, so let $\Delta$ be a triangulation of $Y$ in $CU$ with $d_U(\Delta,o) > M + 5$ by the assumption. Finally by \autoref{Ccsimplycon} there is some $\Delta'$ a triangulation of $X$ with good edges in $CU$ with $d_U(\Delta',o)>M$. 

Combining the triangulation from the homotopy and $\Delta'$ gives the desired triangulation with good edges.


\end{proof}

    


    

\begin{theorem}\label{ozoz}
    Let $\xi(\Sigma)\geq 10$. Suppose for any $\Upsilon$ satisfying $\xi(\Upsilon)=\xi(\Sigma)-1$ and any $B_r\subset C\Upsilon$ that $C\Upsilon\setminus B_r$ is almost simply connected, and suppose any $B_R\subset C\Sigma$ is $N$-almost simply connected. Then for any $r\geq 3$, and loop $\gamma \subset S_r \subset C\Sigma$, there is a triangulation $\Delta\subset B_{r+N+7}\setminus B_{r-3}$ of $\gamma$.
\end{theorem}

\begin{proof} 

Let $\gamma=\{q_0,q_1,...q_n=q_0\}\subset S_r$. $\gamma$ may not have good edges so we will homotope $\gamma$ to a loop with good edges.

\begin{lemma}
    There is a loop $\gamma'$with good edges homotopic to $\gamma$ within $B_2(\gamma)$
\end{lemma}
\begin{proof} \renewcommand{\qedsymbol}{$\blacksquare$}
Suppose $q_i\in\gamma$ is a bad vertex and $L,R$ are the components of $\Sigma\setminus q_i$. If $q_{i-1},q_{i+1} \in CR$ (or $CL$), replace $q_i$ with $q_i'\in CL$ (respectively $CR$). If not, since $\xi(L)+\xi(R)+1=\xi(\Sigma)\geq 10$, WLOG suppose $q_{i-1}\in CR$ and $\xi(R)\geq 5$. Now suppose $R\setminus q_{i-1}$ is two components $D,G$. Since $\xi(D)+\xi(G)+1=\xi(R)\geq 5$, WLOG $\xi(G)\geq 2$. Since $G$ has at most two boundary components, there must be at least 3 punctures on $G$ or a genus so replace $q_{i}$ with a pants curve or a non-separating $q_i'$ on $G$, which is a good vertex, and $q_{i-1},q_i',q_i,q_{i+1}$ is a diamond. Repeating this for all bad vertices in $\gamma$ we get some $\gamma'$ which only has good vertices and is homotopic to $\gamma$ within $B_1(\gamma)$. 


Now suppose $q_i, q_{i+1}$ is a bad edge but both are good vertices. Then consider $\Sigma \setminus (q_i, q_{i+1})$ which has 2 non-pants components $L$ and $R$. By complexity considerations, at least one of the components has complexity at least 4. An essentially non-separating curve in this component will be a good vertex forming good edges with $q_i$ and $q_{i+1}$.

\end{proof}

Since for all $k, B_k$ is $N$-almost simply connected, we may find a triangulation of $\gamma'$ within $B_{r+2+N}$. Then by \autoref{homotopins}, there is a triangulation in $D$ of $\gamma'$ in $B_{r+N+2+5}$ with good edges. Let $v(k, \Delta), e(k, \Delta), f(k, \Delta)$ denote the number of vertices, edges, and faces of $\Delta$ in $S_k$ and let $\sigma(k, \Delta)=v(k, \Delta)+e(k, \Delta)+f(k, \Delta)$. Choose $\Delta$ so that it is a triangulation of $\gamma'$ in $B_{r+N+2+5}$ with good edges and $\sigma(r - 3, \Delta)$ minimal. We show that this is zero, which would show that there is a good-edge triangulation of $\gamma'$ in $B_{r+N+7}\setminus B_{r-3}$ which then gives a triangulation of $\gamma$ in $B_{r+N+7}\setminus B_{r-3}$.

We will prove this by induction on $r$. For the base case, let $r=0$. Suppose for a contradiction that for any $\Delta$ a triangulation of $\gamma'$ with good edges in $B_{r+N+7}$ that $\sigma(0,\Delta)$ is positive. So let $D$ be such a triangulation such that $\sigma(0, \Delta)$ is positive.
Since $S_0$ is a point, only $v(k, \Delta)$ can be non-zero, and all edges of $\Delta$ must have one end not being the origin $o$. Let $\{z_i\}$ be the set of such vertices, as in $z_i=o$ and the link of $z_i$ in $\Delta$ is $Y_i=\{y_{i,j}\}\subset S_1$. There are finitely many triangles $z_i,y_{i,j},y_{i,j+1}$, so we may apply \autoref{fixinglem} finitely many times to the bad triangles in these cones, so that all the cones with respect to $z_i$ are good cones. Then we many apply \autoref{inductivestep!2} to these finitely many good cones to replace them with a triangulation with good edges of the link $Y_i$ in $S_1$. This gives a triangulation $\Delta'$ of good edges of $\gamma'$ in $B_{r+N+7}$ where $\sigma(0,\Delta')=0$, a contradiction.

Now for the inductive step suppose that for some $k\leq r-3$, that $\Delta$ can be chosen so that $\sigma(i-1, \Delta)=0$ for $i=1,...,k$. We show that $\Delta$ can be chosen so that $\sigma(k,\Delta)=0$.
For the following all the triangulations of $\gamma'$ we pick are in $B_{r+N+7}$ and have good edges.\\

\emph{Pushing off faces}: First, suppose $\Delta$ is chosen so that $f(k,\Delta)$ is minimal. We show that $f(k,\Delta)=0$. Suppose not: there is a face of $\Delta$ in $S_k$, say $x,y,z$. Then by \autoref{fixinglem} there's a $w\in S_k\cup S_{k+1}$ so that $(w,x,y),(w,x,z),(w,y,z)$ are good triangles. If $w\in S_{k+1}$ then we replace triangle $(x,y,z)$ with the three triangles $(w,x,y),(w,x,z),(w,y,z)$ and get a contradiction to minimality. If $w\in S_k$ then by \autoref{inductionlem1} there are $a,b,c\in S_{k+1}$ such that $(a,w,x,y), (b,w,x,z),(c,w,y,z)$ are good tetrahedra. Then replacing $x,y,z$ with the nine triangles $(a,x,y),(a,w,y),(a,w,x),(b,x,z),(b,x,w),(b,w,z),(c,y,z),(c,w,y),(c,w,z)$, we get a triangulation $\Delta'$ which has one less face in $S_k$ than $\Delta$, a contradiction to minimality.\\

\emph{Pushing off edges}: Now suppose that $\Delta$ is such that $f(k,\Delta)=0$ and $e(k,\Delta)$ is minimal. Suppose that $e(k,\Delta)$ is positive, and let $(x,y)$ be an edge in $S_k$. Since $(x,y)$ is in $S_k$ it cannot be on the boundary of $\Delta$, and since there are no faces of $\Delta$ in $S_k$, there must be $a,b$ such that $(a,x,y),(b,x,y)$ are triangles in $\Delta$ with $a,b\in S_{k+1}$. If either of these triangles are not good, by \autoref{replacebad1} we may replace this diamond with a collection of at most five diamonds with good triangles. This means we may assume that $(a,x,y),(b,x,y)$ are good triangles, so by \autoref{inductionlem2} there are points $a=x_0,...x_m=b$ so that we can replace the diamond $a,x,y,b$ with triangles $(x,x_i,x_{i+1}),(y,x_i,x_{i+1})$ so that there is one less edge in $\Delta$ in $S_k$, a contradiction.\\

\emph{Pushing off vertices}: Suppose $\Delta$ is such that $f(k,\Delta)=0=e(k,\Delta)$ and $v(k,\Delta)$ is minimal. Suppose for a contradiction that $v(k,\Delta)$ positive for all $\Delta$ such that $f(k,\Delta)=0=e(k,\Delta)$, and let $z\in S_k$ be a vertex of $\Delta$. Since there are no edges of $\Delta$ in $S_k$, the link of $z$ is in $S_{k+1}$. We may assume that all the triangles in the cone are good. For suppose not, and we will find some other triangulation $\Delta'$ with $f(k,\Delta) = 0 = e(k,\Delta)$ and every cone with respect to a vertex in $S_k$ has only good triangles.

So take $z,x,y$ a bad triangle in a cone with respect to $z$, and $x,y\in S_{k+1}$ as in \autoref{badcone}. Then by \autoref{fixinglem}, if the $w$ from the proposition is in $S_{k+1}$ we can replace $z,x,y$ with the three good triangles $(w,x,y),(w,x,z),(w,y,z)$ and the cone will be replaced with a good triangle in $S_{k+1}$ and a cone on $z$ that has one fewer bad triangle. If $w\in S_k$, then $(x,w,z),(y,w,z)$ forms a good diamond so they may be connected by \autoref{inductionlem2} with $x=x_0,...x_m=y$, so that the original cone is replaced with two cones, one of which is the good cone and the other is the original but with the edge $x,y$ replaced by $x_0,...x_m,w$, which has one fewer bad triangle. We repeat this process on every bad triangle to produce a triangulation where all the cones have only good triangles.

Suppose our cones are all of good triangles and $\Delta$ is such that $v(k,\Delta)$ is minimal and positive. Then by \autoref{lem:induction3} and \autoref{inductivestep!2} we may replace the cone with a triangulation in $S_{k+1}$ which contradicts minimality of vertices in $S_k$. Here, we are using the assumption that $C\Upsilon\setminus B_r$ is almost simply connected.

This gives us a triangulation in $B_{r+N+7}\setminus B_{r-3}$ of $\gamma'$ and hence of $\gamma$.
\end{proof}

\section{Base Case}\label{basecase}

Now we must show for any $\xi(\Sigma)\geq 10$ that for any $\Upsilon$ satisfying $\xi(\Upsilon)=\xi(\Sigma)-1$ and any $B_r\subset C\Upsilon$ that $C\Upsilon\setminus B_r$ is almost simply connected, and that any $B_R\subset C\Sigma$ is $N$-almost simply connected for some $N$. So it suffices to show for any $\Sigma$ with $\xi(\Sigma)\geq 9$ and any $B_r\subset C\Sigma$ that $C\Sigma\setminus B_r$ is almost simply connected \autoref{simcon}, and that there is some $N$ depending only on $\Sigma$ so that $B_r$ is $N$-almost simply connected \autoref{ball}.

\begin{theorem} [\autoref{ball}]
    For $\xi(\Sigma)\geq 4$, there is some $M=M(\Sigma)$ such that any $B_r$ is $M$-almost simply connected.
\end{theorem}

We first recall some definitions and results from Teichm\"uller theory.


\begin{definition}
    Let $\Sigma$ be a surface. Let $\calT(\Sigma)$ be the Teichm\"uller space of $\Sigma$. Then the extremal length of a curve $a \in C\Sigma$ with respect to a point $\sigma \in \mathcal{T}(\Sigma)$ is
    \begin{align*}    \Ext_\sigma(a)=\sup_{\sigma'\in[\sigma]}\frac{\inf_{a'\in[a]}\big(\ell_{\sigma'}(a')\big)^2}{\Area(\sigma')}
    \end{align*}
where $[\sigma]$ is the class of metrics which are conformally equivalent to $\sigma$.
    
\end{definition}

\begin{definition}\label{Phi}
    Fix $L > 0$. Let $\Phi=\Phi_L:\mathcal{T}(\Sigma)\to \calP(C\Sigma)$ be the map taking a $\sigma \in\mathcal{T}(S)$ to the set of simple closed curves whose extremal length with respect to $\sigma$ is less than $L$. 
\end{definition}

\begin{proposition}\label{collection}
For any $L>0, \sigma \in\mathcal{T}(\Sigma)$, $\Phi(\sigma)$ is a finite (possibly empty) collection of simple closed curves in $C\Sigma$ with diameter bounded by $2L+1$.
\end{proposition} 
\begin{proof}
    This proposition is a collection of results from ~\cite{bourque21} and ~\cite{MMI}.
Fix $L>0$, we claim the number of curves with extremal length less than $L$ is finite. Suppose for a contradiction that it were infinite, and let $b$ be a curve such that its extremal length is less than $L$. Let $\rho$ be the hyperbolic metric on $\Sigma$ conformally equivalent to $\sigma$. Let $\beta$ be the unique geodesic closed curve in the homotopy class of $b$.
    
    Then \begin{align*}
        L>\Ext_\sigma(b)&\geq\frac{\inf_{b'\in[b]}\ell_\rho^2(b)}{\Area(\rho)}\\
        &= \frac{\ell_\rho^2(\beta)}{\Area(\rho)}\\
        \ell_\rho(\beta)&\leq\Big(\Area(\rho)\Ext_\sigma(b)\Big)^{\frac{1}{2}}\\
        &<\Big(\Area(\rho)L\Big)^{\frac{1}{2}}
    \end{align*}

    This means that there are infinitely many closed geodesic curves with hyperbolic length less than some constant that depends only on $\Sigma$ and $L$, so let $K=\left(\Area(\rho)L\right)^{\frac{1}{2}}$.

    We have an infinite collection of constant speed parametrizations $f_i:S^1\to \Sigma=\Sigma_{g,b}$ whose images are non homotopic geodesic curves with length less than $K$.

    But then this family is a subset of continuous functions between two compact metric spaces (we may remove the cusps of $\Sigma$ since no non-trivial geodesic curve passes within a neighborhood of a cusp), and this family is uniformly bounded and equicontinuous (the equicontinuity follows from the fact that the geodesics have bounded length so $\abs{f_i'} \leq K/2\pi$ hence the $f_i$ are even Lipschitz with the same Lipschitz constant). Thus by Arzela-Ascoli \cite[p.~290]{munkres00} there's a subfamily which converges uniformly. But this means the images of this subfamily eventually lie in an annulus and so are homotopic, a contradiction.

    So finitely many curves have extremal length less than $L$, which means there either are none or there exists a shortest one (possibly more than one).

    For the uniformly bounded diameter, fix $L,\Sigma$ and $\sigma \in\mathcal{T}(\Sigma)$. Then let $a, b \in \Phi(\sigma)$. We have by \cite{MMI} (elementary fact in lemma 2.5) that $i(a, b)^2 \leq \Ext_\sigma(a)\Ext_\sigma(b)<L^2$. Then $d_\Sigma(a,b)\leq 2i(a,b)+1<2L+1$ by ~\cite{MMI} lemma 2.1. 
\end{proof}

As we are building triangulations to construct a homotopy to show simple-connectedness, we define a few concepts related to a triangulation.

\begin{definition}
Let $\Delta=\{\Delta_i\}_{i=1}^n$ be a finite triangulation of the closed disk $D^2$, where $V(\Delta)=\cup_iV(\Delta_i)=\cup_i\{x_i,y_i,z_i\}$ are the vertices, and $E(\Delta)=\cup_iE(\Delta_i)=\\ \cup_i\{[x_i,y_i],[x_i,z_i], [y_i,z_i]\}$ are the edges. 

In the following, $Y$ is $\RR^n$, a curve complex, or Teichm\"uller space and fix some loop $X=\{x_i\}\subset Y$. Then a set $Z\subset Y$ is called a \emph{triangulation} of $X$ if there's some continuous map $f:D^2\to Y$, and $f(S^1)=X,~f(V(\Delta))=Z$.

$Z$ is called a \emph{mesh} of $X$ if there's some continuous map $f:S^1\cup V(\Delta)\to Y$, and $f(S^1)=X, ~f(V(\Delta))=Z$. 

$Z$ is called a \emph{coarse triangulation} of $X$ if there's some continuous map $f:S^1\cup E(\Delta)\cup V(\Delta)\to Y$, and $f(S^1)=X, ~f(V(\Delta))=Z$. 


Given a mesh or coarse triangulation or triangulation and $f$, $\{x,y,z\}\subset Y$ is called a \emph{triple} of $Y$ if there's some $i$ such that $f(V(\Delta_i)) = \{x, y, z\}$. $\{x,y\}\subset Y$ is called a \emph{pair} of $Y$ if there's some $\{a,b\}\in V(E(\Delta))$ such that $f(\{a,b\}) = \{x,y\}$.

The \emph{mesh size} of $Z$ given $f$ is $\max_i\{\diam(f(V(\Delta_i)))\}$, which makes sense for triangulations, coarse triangulations, and meshes. 
\end{definition}

\begin{proposition}\label{triangulate}
    If $(X,d)$ is a simply connected metric space then for any $\epsilon>0$ any loop $\gamma$ there is a triangulation $\Delta$ of $\gamma$ with mesh size $<\epsilon$.
\end{proposition}
\begin{proof}
    Since $X$ is simply connected, let $F:D^2\to X$ be a homotopy $\gamma\simeq_h const$. Since it's a continuous function on a compact set it is uniformly continuous, so we have that there's some $\delta>0$ such that if $|x-y|<\delta$ then $d(F(x),F(y))<\epsilon$. Then take a triangulation $\Gamma$ of $D^2$ with mesh size $<\delta$.
    Look at $F(\Gamma)$ which is a triangulation by definition, and take any triple $\{F(\alpha),F(\beta),F(\gamma)\}$, and note by uniform continuity that $\diam(\{F(\alpha),F(\beta),F(\gamma)\})<\epsilon$ since $\diam(\{\alpha,\beta,\gamma\})<\delta$.
\end{proof}


\begin{definition}
    Let $Thick(\epsilon)=\mathcal{T}_\epsilon\subset\mathcal{T}(\Sigma)$ be the subset of points $\sigma$ in Teichm\"uller where every closed curve on $\sigma$ has extremal length at least $\epsilon$. 
\end{definition}

Markings, and more generally the marking complex, are a useful tool for studying Teichm\"uller space. These were first introduced in \cite{MMII} although \cite{duchin10} also gives a nice summary of results and ideas. 

\begin{lemma}\label{bibddd}
    Pick any marking $\mu = \{(\mu_i, t_i)\}_{i = 1}^{3g - 3 + b}$ and $M>0$. Then $B_\mu=\{\sigma \in \mathcal{T}(S) | \Ext_{\sigma}(\mu_i) \leq M, \Ext_{\sigma}(t_i) \leq M\}$ has bounded diameter. In fact, $B_\mu$ is in $Thick(\epsilon)$ where $\epsilon=\frac{1}{M}$.
\end{lemma}
\begin{proof}
    Suppose $M$ is large enough so that $B_\mu$ is non-empty (for instance we could take $M$ to be Bers' constant).
    The first part is more or less from \cite{duchin10} page 733 or 12. The fact that it has bounded diameter is true for every $M$. Suppose not. Let $\sigma_1, \sigma_2 \in B_\mu$ and $a$ is any simple closed curve. There is some $b$ in the marking that intersects $a$ which means $\Ext_{\sigma_j}(a) \Ext_{\sigma_j}(b) \geq i(a,b)^2 \geq 1$ for $j=1,2$. Then $\frac{\Ext_{\sigma_1}(a)}{\Ext_{\sigma_2}(a)} \leq \frac{\Ext_{\sigma_2(b)}}{\Ext_{\sigma_1(b)}} \leq M^2$. Hence $B_\mu$ has bounded diameter with this bound only depending on $M$ (and not on the marking).

    For the second part, fix $\sigma\in B_\mu$ and suppose there is a curve $a\in C\Sigma$ such that $\Ext_\sigma(a)<\epsilon$. Since $\mu$ is a complete marking, $a$ intersects some base curve or transverse curve in the marking, say $b\in \mu$. Then $\Ext_x(b)\Ext_x(a)\geq i(a,b)^2\geq 1$ gives that $\Ext_x(b)\geq\frac{1}{\Ext_x(a)}>\frac{1}{\epsilon}=M$, a contradiction.
\end{proof}



\begin{lemma}\label{closebythick}
    Let $S$ be a surface and $\mathcal{T}(S)$ be its Teichm\"uller space. If $x\in Thick(\epsilon)$ and $d_\mathcal{T}(x,y)\leq M$ then $y\in Thick(\epsilon')$ where $\epsilon'= e^{-2M}\epsilon$.
\end{lemma}

\begin{proof}
    Suppose not, and there is a simple closed curve $\gamma\subset S$ with $\Ext_y(\gamma)<\epsilon'$. Then we have
    \begin{align*}
        d_\mathcal{T}(x,y)=d_\mathcal{T}(y,x)&=\frac{1}{2}\log\sup_{\alpha\in C_0}\frac{\Ext_x(\alpha)}{\Ext_y(\alpha)}\\
        &\geq \frac{1}{2}\log\frac{\Ext_x(\gamma)}{\Ext_y(\gamma)}\\
        M&\geq\frac{1}{2}\log\frac{\Ext_x(\gamma)}{\Ext_y(\gamma)}\\
        {\Ext_y(\gamma)}e^{2M}&\geq \Ext_x(\gamma)>\epsilon\\
    \end{align*} 

a contradiction.
\end{proof}



The following lemma shows that if surfaces are close in Teichm\"uller space, then their images in $C\Sigma$ have bounded intersection number.
\begin{lemma}\label{boundint}
    If $d_{\calT}(x,y)< B$ then $i(\{\Phi(x)\cup\Phi(y)\})<N=N(B,\Sigma)$.
\end{lemma}

 \begin{proof}
The key idea is that if two curves intersect many times then at least one of them must be quite long due to the collar lemma.

Pick any $\alpha_1,\alpha_2\in \{\Phi(x)\cup\Phi(y)\}$.
Suppose $\alpha_1\in \Phi(x), \alpha_2\in\Phi(y)$, we wish to bound their intersection number (if both $\alpha_i$ lie in exactly one of $\Phi(x)$ or $\Phi(y)$ then we are done by . By definition of $\Phi$, this means $\Ext_x(\alpha_1)=\inf_{\alpha} \Ext_x(\alpha)\leq e_0=e_0(\Sigma)$, and similarly, $\Ext_y(\alpha_2)=\inf_{\alpha} \Ext_y(\alpha)\leq e_0=e_0(\Sigma)$. Here, $e_0(\Sigma)$ is a number depending on $\Sigma$ such that all extremal length systoles of any curve and metric on $\Sigma$ is not more than $e_0(\Sigma)$. For example, we could take $e_0(\Sigma)$ to be the Bers' constant for $\Sigma$.

We have by (2.1) on page 4 of \cite{MMI}

 \begin{align*}
    B>d_T(x,y)&=\frac{1}{2}\log\sup_{\alpha\in C_0}\frac{\Ext_y(\alpha)}{\Ext_x(\alpha)}\\
    &\geq \frac{1}{2}\log\frac{\Ext_y(\alpha_1)}{\Ext_x(\alpha_1)}
\end{align*}

So we get $\Ext_y(\alpha_1)<e^{2B} \Ext_x(\alpha_1)\leq e^{2B} e_0$

Then by the proof of Lemma 2.5 (page 6 in \cite{MMI}), since we have both $\Ext_y(\alpha_1)<e^{2B}e_0$ and $\Ext_y(\alpha_2)<e_0$ by picking $N=e^{2B} \cdot e_0$ we get $i(\alpha_1,\alpha_2) \leq e^{2B} e_0$. \end{proof}

\begin{proposition}\label{fact}
    We recall a fact: $d_\Sigma(x,y)\leq 2\log i(x,y)+2$. By above lemma, we note that if $d_\calT(x,y)<B $, then $d_\Sigma(\Phi (x),\Phi(y))\leq 2\log i(\Phi (x),\Phi(y))+2\leq 2\log (e^{2B}e_0)+2=4B+2\log e_0+2$
\end{proposition}

\begin{theorem}\label{bosmcsc1}
    Fix a surface $\Sigma$. Then there exists a constant $K=K(\Sigma)$ such that for any loop $X\subset C\Sigma$, there is a mesh $\Delta^*$ of $X$ such that for each triple $\{\beta_1,\beta_2,\beta_3\}$ of $\Delta^*$, $i(\beta_i,\beta_j)<K$ and $d_\Sigma(\Delta^*,X)< K\cdot (\diam (X)+1)$.
\end{theorem}

\begin{proof}
    Fix $L>0$ and let $\Phi=\Phi_L:\mathcal{T}(\Sigma)\to\mathcal{PC}(\Sigma)$ be the map from definition \autoref{Phi}. Our main idea is to `pull back' $X$ to a loop in $\calT(\Sigma)$, triangulate the loop there and then push forward this triangulation back to $C\Sigma$. In order to begin therefore, for each $\alpha_i \in X$, we want that there's some $\sigma_i$ in the $\epsilon$-thick part of $\mathcal{T}(\Sigma)$ such that $\alpha_i\in \Phi(\sigma_i)$, where $\epsilon$ only depends on $L$.

    For each edge $\{\alpha_i, \alpha_{i + 1}\}$, construct a marking $\mu_i$ which contains $\alpha_i, \alpha_{i+1}$ as base curves. Construct a marking path $\{\mu_{i, j}\}$ from $\mu_i$ to $\mu_{i+1}$ as in \cite{MMII}. For each $\{\mu_{i, j}\}$ define $B_{i, j} := \{\sigma \in \mathcal{T}(\Sigma) | \ell_\sigma(\mu_{i, j}) \leq L\}$. Then the marking path defines a sequence of points in $\mathcal{T}(\Sigma)$ by sending $\mu_{i, j}$ to any point in $B_{i, j}$. Label this sequence $\{\sigma_j\}$ which lies in $\mathcal{T}_\epsilon$ for $\epsilon = 1/L$ by \autoref{bibddd}. Since consecutive $\mu_{i, j}$ differed by elementary moves, it follows that $d_{\mathcal{T}}(\sigma_j, \sigma_{j+1})$ is bounded as was shown in \cite{duchin10}. Let $\gamma$ be the loop in $\mathcal{T}(\Sigma)$ formed by connecting $\sigma_j$ to $\sigma_{j+1}$ via Teichm\"uller geodesics. This is the `pull back' of $X$ we sought.

    Now we triangulate the loop in $\calT(\Sigma)$.
    Let $g_j$ be the Teichm\"uller geodesic from $\sigma_0$ to $\sigma_j$. By Theorem C in \cite{rafi14}, $g_j$ fellow travels $g_{j-1}, g_{j+1}$.
    Denote by $\gamma^i$ the segment on $\gamma$ between $\sigma_i$ and $\sigma_{i+1}$. Then $g_j,\gamma^j,g_{j+1}$ form the edges of a geodesic triangle in $\mathcal{T}$, which has the property that the two edges $g_j,g_{j+1}$ are not more than $D$ apart. By \autoref{triangulate} there's a triangulation $\Delta_j$ of this triangle with mesh size $<1$, where every vertex of $\Delta_j$ is within distance $D$ of one of the vertices on the edges $g_j,g_{j+1}$. Let $\Delta=\cup_j\Delta_j$. We can decompose $\Delta$ into $\Delta=\{e_i\}\sqcup\{f_j\}$ where $e_i$ are points on the boundary $\gamma$ and $f_j$ are not. To be sure, we can have it so that $\{\sigma_i\} \subset \{e_i\}$. 
    
    We can use $\Phi$ to send $\Delta$ to get a mesh in $C\Sigma$. For each $i,j$, pick points $\eta_i,\nu_j$ from $\Phi(e_i),\Phi(f_j)$, such that the $\{\alpha_i\}\subset\{\eta_i\}$. For each edge $[\delta_i,\delta_j]$ in $\Delta$, suppose $\psi_i,\psi_j$ are the points we picked from $\Phi(\delta_i),\Phi(\delta_j)$. Then fix a geodesic in the curve complex $[\psi_i,\psi_j]$. Now $\cup_{[\delta_i,\delta_j]\in E(\Delta)}[\psi_i,\psi_j]$ gives a coarse triangulation $\Delta^*$ of $X$. Actually, we will forget the edges of this coarse triangulation so that $\Delta^*$ is the mesh we seek.\\
    
    We verify that this triangulation $\Delta^*$ satisfies all the properties in the theorem. Pick any triple $\{x,y,z\}\subset \Delta^*$. This triple is in the image of $\Phi(\xi),\Phi(\zeta),\Phi(\rho)$ where $\xi,\zeta,\rho$ were a triple in $\Delta$. Since $\Delta$ had mesh size $<1$, $\diam(\{\xi,\zeta,\rho\})<1$. Then by \autoref{boundint}, $i(\{x,y,z\})<N(1,\Sigma)$.

    Next we need to compute a bound for $d_{\Sigma}(a_0, b)$ where $a_0 \in X$ and $b \in \Delta^*$. First we verify this for $b \in \Delta^*$ that lie in $\Phi(g_j)$. By \cite{MMI}, we know that $\Phi$ takes geodesics to quasigeodesics and hence there exist constants $\kappa_1, \kappa_2$ such that
    $$d_\Sigma(b,a_0)\leq \diam_\Sigma(\Phi(g_j))\leq \kappa_1 \diam(X) + \kappa_2$$

    Now suppose we have an arbitrary $b \in \Delta^*$. Let $\beta$ be the point in $\Delta$ so that $b = \Phi(\beta)$. By construction we have that there is some $g_j$ so $d_{\calT}(\beta, g_j) < D$. So let $\lambda$ be the point in $g_j \cap V(\Delta)$ so that $d_{\calT}(\beta, \lambda) < D$. Thus we can apply \autoref{boundint} to get $d_{\Sigma}(\Phi(\beta), \Phi(\lambda)) < N(D, \Sigma)$. Combining this with the previous bound and using the triangle inequality we get 
    $$ d_{\Sigma}(b, X) \leq d_{\Sigma}(b, a_0) \leq \kappa_1 \diam(X)+\kappa_2 + N(D,\Sigma) $$
    

    

    
    Our theorem follows when we pick $K=\max\{N(1,\Sigma),\kappa_1, \kappa_2 + N(D,\Sigma)\} $. This can be made into a function of only $\Sigma$ by picking $L$ to be the Bers constant which depends only on $\Sigma$ so that $\Phi(x)\neq \emptyset,~\forall x\in\mathcal{T}$, see \cite[Theorem 12.8]{Farb11}.
\end{proof}

\begin{figure}[h]
\centering
\includegraphics[width=1\linewidth]{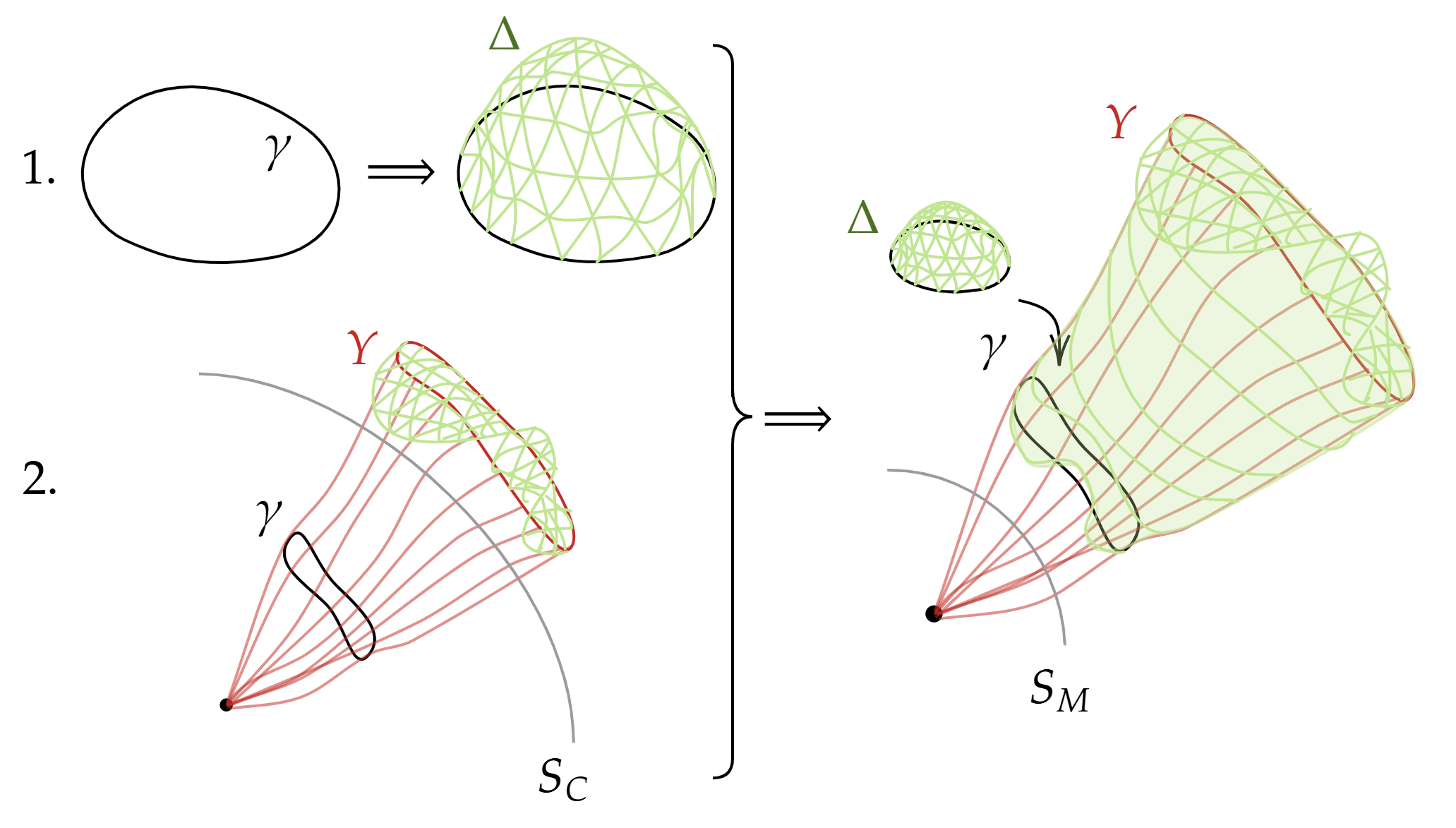}
\caption{Plan for base case}
\end{figure}   

\begin{lemma}\label{coollemma}
    Let $\xi(\Sigma) \geq 4$. Then there is some constant $L=L(\Sigma)$ such that for any loop $X\subset C\Sigma$ there is a triangulation $\Delta$ of $X$ with $d_\Sigma(\Delta,X) < L(diam(X)+1)$.
\end{lemma}
\begin{proof}
    Let $\Sigma$ and $K$ as in \autoref{bosmcsc1}. 

    There are finitely many pairs of curves up to homeomorphism of intersection number less than $K$. For any such pair $x,y$, write $Vert(x,y)$ to mean the orbit $Mod(\Sigma)\cdot\{x,y\}\subset C\Sigma$. Pick one pair $\{x,y\}\in Vert(x,y)$. Let $\gamma_{xy}$ be a geodesic (of length $\leq 2\log K +2$). Write $Edge(x,y)$ to mean $Mod(\Sigma)\cdot \gamma_{xy}$.  
    
    Let $T_K$ be the set of triples $\{x,y,z\}$ where $i(\{x,y,z\})<K$. Define $$Triple(x,y,z)=Mod(\Sigma)\cdot\{x,y,z\}$$
    Pick a triple $\{x,y,z\}\in Triple (x,y,z)$. Now by above, we have some unique $\gamma_{xy}\in Edge(x,y)$ that is a geodesic from $x$ to $y$.
    
    Let $D_{x,y,z}$ be a triangulation of the loop $\gamma_{xyz}$ made from the concatenation of $\gamma_{xy},\gamma_{xz},\gamma_{zy}$. This exists because $C\Sigma$ is simply connected for $\xi(\Sigma) \geq 4$. Write $Disk(x,y,z)=Mod(\Sigma)\cdot D_{x,y,z}$. We have that because the $Mod(\Sigma)$ action is an isometry, all the elements of $Disk(x,y,z)$ have the same diameter, so $\diam (Disk (x,y,z))$ is well defined. Since there are finitely many such sets $Disk (x,y,z)$, Let $\kappa=\max_{T_K} \diam (Disk (x,y,z))$, which is a constant that only depends on $K$ and the surface $\Sigma$.

    Now pick $X$ any loop in $C\Sigma$ and $\Delta^*$ a mesh as in \autoref{bosmcsc1}.
    For every pair $\{x,y\}\subset \Delta^*$, there's some unique $\gamma_{xy}\in Edge(x,y)$. Taking the union $\cup_{xy~pair~in~\Delta^*}~\gamma_{xy}$ gives us a coarse triangulation of $X$ since for adjacent points in $X$ we have that $\gamma_{xy}$ is the edge between them that is in $X$. Now for every loop $\gamma_{xyz}$ for triples $x,y,z\subset \Delta^*$, there's a unique $D_{x,y,z}\in Disk(x,y,z)$ that triangulates $\gamma_{xyz}$. Thus the union $\cup_{xyz~triple~in~\Delta^*}D_{x,y,z}$ is a triangulation $\Delta$ of $X$. Now, since these disks all have diameter $\leq\kappa$, we get that $d_\Sigma(\Delta,\Delta^*)\leq\kappa$. Also, $d_\Sigma(\Delta^*,X)\leq K(\diam(X)+1)$, so by the triangle inequality, $d_\Sigma(\Delta,X)\leq K(\diam(X)+1)+\kappa$, so we can take $L=K+\kappa$ and we are done.
    
\end{proof}

\begin{lemma}\label{importantlemma}
    Let $\xi(\Sigma) \geq 4$, $o\in C\Sigma$, $M>0$, and take any loop $X=\{x_0,...,x_n=x_0\}\subset C\Sigma$. For all $i$ let $\{\xi_i(t)\}_{t=1}^\infty$ be a geodesic ray starting at $o$ and passing through $x_i$. Then there is a loop $Y$, a family of geodesic rays $\{\gamma_j(t)\}_{t=1}^\infty \supset\{\xi_i(t)\}_{t=1}^\infty$, and some $r$ such that $\gamma_j(r)\in Y$ for all $j$, and a triangulation $\Delta$ of $Y$ such that $d_\Sigma(\Delta,o)>M$, and $\gamma_i,\gamma_{i+1}$ are $\delta$-fellow traveling up to $r$. 
\end{lemma} 

\begin{proof}
The basic outline of the proof is as follows: we need to triangulate $X$ (or some homotopy of it) with points far away from the origin $o$. We can do so using the simply connectedness of the boundary of $C\Sigma$ and the hyperbolicity of $C\Sigma$. In particular, we can use geodesic rays to push $X$ to the boundary and triangulate it there. The visual metric on the boundary ensures that if points are nearby in the boundary then they fellow-travel for a long time. Thus by having a sufficiently fine triangulation on the boundary, we can get a mesh of points in $C\Sigma$ for $X$ with small mesh size that are all far from $o$. Since the mesh size is small, we can fill the triples in this mesh with small disks. All of this is made more precise below.\\

Let $o,X,\Sigma,M$, and
$\xi_i(t)$ as above. The geodesic rays $\xi_i(t)$ define points $\xi_i\in \partial C\Sigma$. Since $\partial C\Sigma$ is path connected, we may find paths $[\xi_i,\xi_{i+1}]\subset \partial C\Sigma$ and concatenate them to form a loop $\Xi \subset \partial C \Sigma$. We also have that $\partial C\Sigma$ is a simply connected metric space, with a visual metric $d$ with parameter $a\in (0,1)$. This visual metric has the property that for any $E$ large enough, there's some $C>0$ such that if $r$ is the maximal number such that $d_\Sigma(\eta(r),\zeta(r))\leq E$, then $\frac{1}{C}a^r\leq d(\eta,\zeta)\leq Ca^r$, where $\eta(t),\zeta(t)$ are geodesic rays from $o$. Fix $E,C$ so that these estimates work. What this means is that if $d(\eta,\zeta)<\frac{1}{C}a^p$ then we would have some $r$ which is the maximal number such that $d_\Sigma(\eta(r),\zeta(r))\leq E$, and 
$$\frac{1}{C}a^r\leq d(\eta,\zeta)<\frac{1}{C}a^p$$
which implies $r>p$. 

Fix $K$ as in \autoref{bosmcsc1}.
There are finitely many pairs of curves up to homeomorphism of intersection number less than $K$. For each of those pairs $\{x,y\}\subset C\Sigma$ fix a path $[x,y]$ in $C\Sigma$ (which should be the same path up to mapping class group action on $C\Sigma$), and for each of those triples, fix a disk $D_{x,y,z}$ and a triangulation of it, by simple-connectedness of $C\Sigma$. Let $\kappa=\max_{x,y,z} \diam (D_{x,y,z})$. 

Let $$\omega=E+M+K+\kappa+\delta+1$$ By \autoref{triangulate}, we can find a triangulation $\Delta$ of $\Xi$ with mesh size $<1/C a^L$ for some
$L>100\omega$.
We may assume that $\{\xi_i\}\subset V(\Delta)$. So this means for each edge $\{\xi,\eta\}\in E(\Delta)$, by the remark above, we have that there's some $r>L$ such that $d_\Sigma(\xi(r),\eta(r))\leq E$. Now by choice of $L$ and the hyperbolicity of $C\Sigma$, $\xi,\eta$ are $\delta$-fellow traveling for at least $t\in [0,99\omega]$. Let $s = 50\omega$, and look at $\{\xi(s)\}_{\xi\in V(\Delta)}$. For each $\{\xi,\eta\}\in E(\Delta)$, the geodesic in the curve complex $g_{\xi,\eta}=[\xi(s),\eta(s)]$ is such that $d_\Sigma(o,g_{\xi,\eta})> 48\omega$. Let $g_e=\cup_{\xi,\eta}g_{\xi,\eta}$ which is a loop. 

For each triple $\{\xi,\eta,\zeta\}\in \Delta$, we have that $g_{\xi,\eta},g_{\xi,\zeta},g_{\zeta,\eta}$ form a loop $g_{\xi,\eta,\zeta}$, where $\diam(g_{\xi,\eta,\zeta})\leq 2E$. Then by \autoref{bosmcsc1} there's a mesh $\Delta_{\xi,\eta,\zeta}$ of $g_{\xi,\eta,\zeta}$ with $d_\Sigma(\Delta_{\xi,\eta,\zeta},g_{\xi,\eta,\zeta})< 2\kappa$ for the same $K$ as above. So then taking the union of all these coarse triangulations, let $\Delta_e=\cup_{\xi,\eta,\zeta} \Delta_{\xi,\eta,\zeta}$. which is a coarse triangulation of $g_e$ and $d_\Sigma(\Delta_e,o)> 44\omega$. 

Now, every triple $\{\xi,\eta,\zeta\}$ of $\Delta_e$ is such that $i(\{\xi,\eta,\zeta\})<K$ by \autoref{bosmcsc1}. For each triple $\{\xi,\eta,\zeta\}$, by the mapping class group action we may find some $D_{x,y,z}$ with $\diam (D_{x,y,z})\leq\kappa$ and $\phi\in Mod(\Sigma)$ such that $V(\phi(D_{x,y,z}))=\{\xi,\eta,\zeta\}$. Denote $D_{\{\xi,\eta,\zeta\}}=\phi(D_{x,y,z})$. Then the union of all these triangulations $D_{\{\xi,\eta,\zeta\}}$ is a triangulation $\Gamma$ of some loop $Y$ which contains all the boundary vertices $\{\xi(r)\}_{\xi\in V(\Delta)}$. In particular, all the $\{\xi_i(s)\}\subset Y$ where $\xi_i$ were the rays that pass through $x_i$. By above, we get that $\xi_i,\xi_{i+1}$ are $\delta$-fellow traveling up to time $s$.

Also, $d_\Sigma(\Gamma,o)> 43\omega > M$, and of course $d_\Sigma(Y,o)>M$ since $Y\subset \Gamma$. 

\end{proof}

\begin{theorem}\label{simcon}
    For $\xi(\Sigma)\geq 4$, there is some $M$ such that for any $R$, any loop $X\subset C\Sigma\setminus B_{R+M}$ may be triangulated by some $\Delta\subset C\Sigma \setminus B_{R}$.
\end{theorem}
\begin{proof}
    Fix $o$ the origin for $C\Sigma$, any constant $R$, and let $M> L(2\delta+2) $ where $L$ is from \autoref{coollemma} and $\delta$ is the hyperbolicity constant.
    
    Let $X\subset C\Sigma \setminus B_{R+M}$. By \autoref{importantlemma}, there's a loop $Y=\{y_i\}$ and a triangulation $\Gamma\subset CS\setminus B_{R+2M}$ of $Y$. Label $y_i=\xi_i(r)\in Y$ for each $x_i\in X$. 

    Look at the loop $\alpha$ defined by the concatenation $[x_i,y_i],[y_i,y_{i+1}], [y_{i+1},x_{i+1}],[x_i,x_{i+1}]$ where these $[x_i,y_i]$ are the geodesics from \autoref{importantlemma} and $[y_i,y_{i+1}]$ is a segment of $Y$. 
    It suffices to construct a triangulation $T_{i}$ of these loops with $d_\Sigma(T_{i},o)>R$, for then we would take the union of these triangulations together with $\Gamma$ to form $\Delta$.

    Say $a,b,c,d$ are geodesic segments and the concatenation form a loop $[a,b,c,d]$. Suppose $a$ has length 1, $c$ has length $\leq \delta$ the hyperbolicity constant of $C\Sigma$ and $b,d$ are the same length. We want a triangulation $T$ with $d_\Sigma(T,[a,b,c,d])<M$ where $M$ depends on $\Sigma$. Note that the above loops satisfy this by their construction in \autoref{importantlemma}.

    By the thin triangles condition, $b,d$ are $2\delta$ fellow traveling. So parameterize $b=b(t), d=d(t)$ where $d_S(d(t),b(t))<2\delta$. Now fix geodesics $[b(t),d(t)]$ for each $t$. The loop obtained by concatenating $[b(t),d(t)],[d(t),d(t+1)],[b(t+1),d(t+1)][b(t),b(t+1)]$ has diameter $<2\delta+1$ so may be triangulated by a triangulation with diameter $<L(2\delta+2)$ by \autoref{coollemma}. Then we are done since $M> L(2\delta+2)$.
\end{proof}

Now we show that there is some $N$ depending only on $\Sigma$ so that $B_r$ is $N$-almost simply connected.

\begin{theorem}\label{ball}
    There is some $N=N(\Sigma)$ such that for any origin $o$ and $r$, $B_r(o)$ is $N$-almost simply connected.
\end{theorem}

\begin{proof}
    Fix $K$ and let $L(K,\Sigma)$ be the constant from \autoref{coollemma}.

    Fix any loop $\{x_0,...,x_n\}=\gamma\subset B_r$ and for each $i$ let $\gamma_i$ be a geodesic from $x_0$ to $x_i$. It suffices to show there is some $N$ such that any loop $\gamma_i*[x_i,x_{i+1}] * \gamma_{i+1}^{-1}$ may be triangulated inside of $B_{r+N}$.

    So take $x_0,x_i,x_{i+1}$ and let $\gamma_i=\{x_0=a_0,a_1,...a_l=x_i\}$ and $\gamma_{i+1}=\{x_0=b_0,...,b_k=x_{i+1}\}$ where $|k-l|\leq 1$. By hyperbolicity, $d(a_j,b_j)\leq 2\delta$. This is because hyperbolicity implies that there is some $b^*$ on $\gamma_{i+1}$ such that $d(a_j, b^*) \leq \delta$. Since $b^*$ is on the geodesic $\gamma_{i+1}$ we have $d(b^*, b_j) = \abs{d(x_0, b_j) - d(x_0, b^*)}$. By definition $d(x_0, b_j) = j$ and
    $$d(x_0, b^*) \leq d(x_0, a_j) + d(a_j, b^*) \leq j + \delta$$
    Thus $d(b_j, b^*) \leq \delta$. Putting this all together we get 
    $$d(a_j, b_j) \leq d(a_j, b^*) + d(b^*, b_j) \leq 2\delta$$
    
    Now let $[a_j,b_j]$ be a geodesic from $a_j$ to $b_j$ and look at the loop
    $$ l_j = [a_j, b_j] * [b_j, b_{j+1}] * [b_{j+1}, a_{j+1}] * [a_{j+1}, a_j]$$
    which has diameter at most $2\delta+2$. Then we may triangulate the loop with a triangulation $\Delta$ with $d(\Delta,l_j)< L(2\delta+3)$ by \autoref{coollemma}. Since each of these loops $l_j\subset B_r$, we get a triangulation $\Delta'$ of $\gamma$ with $\Delta'\subset B_{r+L(2\delta+3)}$ as needed.
\end{proof}

The constant $L$ depends on the values from the map $\Phi$, the hyperbolicity constant of $C\Sigma$, and also the finitely many fixed triangulations from \autoref{coollemma}. We conjecture that it does not depend on the surface $\Sigma$.

\bibliographystyle{plain}

\end{document}